\documentclass[11pt]{article}
\usepackage{fullpage}
\usepackage{amsmath,amssymb,amsthm}
\usepackage[english]{babel}
\usepackage[left=3cm,right=3cm,top=3cm,bottom=4cm]{geometry}

\usepackage[final]{showlabels}

\usepackage{tikz,tikz-cd}
\usetikzlibrary{positioning}
\usepackage{braket}
\usepackage{mathtools,stmaryrd}

\SetSymbolFont{stmry}{bold}{U}{stmry}{m}{n}
\usepackage{lmodern,bm,bbm,mathrsfs}
\usepackage{manfnt}

\usepackage{enumitem}

\usepackage[scr=esstix]{mathalfa}

\usepackage{hyperref}
\usepackage[textsize=tiny,colorinlistoftodos]{todonotes}

\theoremstyle{plain}
\newtheorem*{theorem}{Theorem}
\newtheorem*{proposition}{Proposition}
\newtheorem*{lemma}{Lemma}
\newtheorem*{corollary}{Corollary}
\theoremstyle{definition}
\newtheorem*{definition}{Definition}
\newtheorem*{example}{Example}
\newtheorem*{remark}{Remark}

\renewcommand{\bar}[1]{\overline{#1}}
\newcommand{\cat}[1]{\mathscr{#1}}
\renewcommand{\hat}[1]{\widehat{#1}}

\renewcommand{\tilde}[1]{\widetilde{#1}}
\renewcommand{\vec}[1]{\bm{#1}}
\newcommand{\bC}{\mathbb{C}}
\newcommand{\bE}{\mathbb{E}}
\newcommand{\bH}{\mathbb{H}}

\newcommand{\bP}{\mathbb{P}}
\newcommand{\bQ}{\mathbb{Q}}

\newcommand{\bZ}{\mathbb{Z}}
\newcommand{\cE}{\mathcal{E}}
\newcommand{\cF}{\mathcal{F}}
\newcommand{\cK}{\mathcal{K}}
\newcommand{\cL}{\mathcal{L}}
\newcommand{\cN}{\mathcal{N}}
\newcommand{\cO}{\mathcal{O}}
\newcommand{\cT}{\mathcal{T}}
\newcommand{\fM}{\mathfrak{M}}
\newcommand{\fN}{\mathfrak{N}}
\newcommand{\sA}{\mathsf{A}}
\newcommand{\sE}{\mathsf{E}}
\newcommand{\sG}{\mathsf{G}}
\newcommand{\sS}{\mathsf{S}}
\newcommand{\sT}{\mathsf{T}}

\newcommand{\loc}{\mathrm{loc}}
\newcommand{\sst}{\mathrm{sst}}
\newcommand{\pt}{\mathrm{pt}}
\newcommand{\vir}{\mathrm{vir}}

\DeclareMathOperator{\Bl}{Bl}
\DeclareMathOperator{\Char}{char}
\DeclareMathOperator{\cochar}{cochar}
\DeclareMathOperator{\ch}{ch}

\DeclareMathOperator{\DT}{\mathsf{DT}}
\DeclareMathOperator{\Ell}{Ell}
\DeclareMathOperator{\Ext}{Ext}
\DeclareMathOperator{\ext}{ext}
\DeclareMathOperator{\Fr}{Fr}
\DeclareMathOperator{\GL}{GL}

\DeclareMathOperator{\Hom}{Hom}

\DeclareMathOperator{\im}{im}
\DeclareMathOperator{\Jac}{Jac}

\DeclareMathOperator{\Pic}{Pic}
\DeclareMathOperator{\Poles}{\mathsf{Poles}}
\DeclareMathOperator{\rank}{rank}
\DeclareMathOperator{\res}{\mathsf{res}}
\DeclareMathOperator{\Res}{Res}

\DeclareMathOperator{\SU}{SU}
\DeclareMathOperator{\supp}{supp}
\DeclareMathOperator{\Spec}{Spec}
\DeclareMathOperator{\Split}{Split}

\DeclareMathOperator{\tot}{tot}
\DeclareMathOperator{\tr}{tr}
\DeclareMathOperator{\VW}{\mathsf{VW}}
\DeclareMathOperator{\wt}{wt}
\DeclarePairedDelimiter{\inner}{\langle}{\rangle}
\DeclarePairedDelimiterX{\pseries}[1]{[}{]}{\mkern-2mu\delimsize[#1\delimsize]\mkern-2mu}

\title{Invariance of elliptic genus under wall-crossing}
\author{Henry Liu}
\date{\today}

\begin{document}

\maketitle

\begin{abstract}
  Wall-crossing formulas for various flavors of elliptic genus can be
  obtained using master spaces. We give a topological criterion which
  implies that such wall-crossing formulas are trivial. Applications
  are given for: GIT quotients, following Thaddeus; moduli of sheaves,
  following Mochizuki; Donaldson--Thomas and Vafa--Witten
  theory, following Joyce and Tanaka--Thomas respectively.
\end{abstract}

\tableofcontents

\section{Introduction}

\subsection{}

Let $X$ be a smooth proper scheme over $\bC$. Recall the
$q$-Pochhammer and (normalized) odd Jacobi theta functions
\[ \phi(z) \coloneqq \prod_{n > 0} (1 - q^n z), \qquad \vartheta(z) \coloneqq (1 - z^{-1}) \phi(qz) \phi(qz^{-1}), \]
respectively. Both can be extended to functions of vector bundles on
$X$ as
\[ \Phi(\cE) \coloneqq \prod_{\cL \in \cE} \phi(\cL), \qquad \Theta(\cE) \coloneqq \prod_{\cL \in \cE} \vartheta(\cL), \]
where the products range over the Chern roots $\cL$ of $\cE$. Then,
following \cite{Krichever1990,Hoehn1991}, the {\it elliptic genus} of
$X$ is
\begin{equation} \label{eq:elliptic-genus}
  \sE_{-y}(X) \coloneqq \chi\left(X, \frac{\Theta(y\cT_X)}{\Phi(\cT_X) \Phi(\cT_X^\vee)}\right).
\end{equation}
where $\cT_X$ is the tangent bundle and $y$ is a formal variable. If
$\sG$ is a group of automorphisms of $X$, then
\eqref{eq:elliptic-genus} is naturally an element
\[ \sE_{-y}(X) \in K_{\sG}(\pt)[y^{\pm 1}]\pseries*{q}, \]
where, for a scheme $Z$ with $\sG$-action, $K_\sG(Z)$ denotes the
$\sG$-equivariant K-theory of $Z$, i.e. the Grothendieck group of
$\sG$-equivariant coherent sheaves on $Z$.

Note that the $q=0$ specialization of $\sE_{-y}(X)$ is the Hirzebruch
$\chi_{-y}$ genus of $X$.

\subsection{}

To give a sense of what elliptic genus looks like, consider the simple
case where a torus $\sT = (\bC^\times)^r$ acts on $X$ and the fixed
point set $X^\sT$ consists only of isolated points. Then, by
torus-equivariant localization \cite{Thomason1992},
\begin{equation} \label{eq:elliptic-genus-localized}
  \sE_{-y}(X) = \sum_{p \in X^\sT} \prod_{w \in T_pX} \frac{\vartheta(yw)}{\vartheta(w)}
\end{equation}
where the product ranges over the $\sT$-weights of $T_pX$. Indeed, the
definition \eqref{eq:elliptic-genus} is purposefully designed so that
\eqref{eq:elliptic-genus-localized} holds.

More generally, let $\cN_{X^\sT \subset X}$ be the normal bundle of
$X^\sT \subset X$. Then localization, combined with the
multiplicativity of $\Theta$ and $\Phi$, produces
\begin{equation} \label{eq:elliptic-genus-localized-general}
  \sE_{-y}(X) = \chi\left(X^\sT, \frac{\Theta(y \cN_{X^\sT \subset X})}{\Theta(\cN_{X^\sT \subset X})} \frac{\Theta(y\cT_{X^\sT})}{ \Phi(\cT_{X^\sT}) \Phi(\cT_{X^\sT}^\vee) } \right).
\end{equation}
Note that only the first factor in the integrand has non-trivial
$\sT$-dependence.

\subsection{}
\label{sec:elliptic-genus-is-elliptic}

It turns out $\sE_{-y}(X)$ has better properties when it is valued in
equivariant {\it elliptic cohomology}, rather than equivariant
K-theory as written in \eqref{eq:elliptic-genus}. For a torus $\sT =
(\bC^\times)^r$, this means being a section of a line bundle on the
elliptic cohomology scheme \cite{Grojnowski2007}
\[ \Ell_{\sT}(\pt) \coloneqq \sT/q^{\cochar \sT} \cong (\bC^\times/q^\bZ)^r, \]
rather than merely on the K-theory scheme $\Spec K_\sT(\pt) = \sT$.
Here $q = \exp(2\pi i\tau)$ where $\tau$ is the elliptic modulus,
$\bC^\times/q^{\bZ}$ is the Tate elliptic curve, and $\cochar \sT$ and
$\Char \sT$ denote the cocharacter and character lattices of $\sT$.

Given a line bundle $\cL$ on $X^\sT$, let $\wt_\sT(\cL) \in \Char \sT$
correspond to its $\sT$-weight. Using
\eqref{eq:elliptic-genus-localized-general} and the basic
$q$-difference equation $\vartheta(qz) = -(qz)^{-1} \vartheta(z)$, one
can check that if
\begin{equation} \label{eq:elliptic-condition}
  y^{-\sum_{\cL \in \cN_{X^\sT \subset X}} \inner{\sigma, \wt_\sT(\cL)}} \text{ is constant on } X^\sT
\end{equation}
for any cocharacter $\sigma \in \cochar \sT$, then $\sE_{-y}(X)$
satisfies a constant-coefficient $q$-difference equation and therefore
descends from $\sT$ to $\sT/q^{\cochar\sT}$.

\subsection{}
\label{sec:y-specialization}

The simplest way to satisfy \eqref{eq:elliptic-condition} is to assume
that $\det \cT_X$ is an $N$-th power, so that $\sum_\cL \wt_\sT(\cL)$
is a multiple of $N$, and to specialize $y$ to an $N$-th root of unity
$\zeta_N \neq 1$. Thus the following two cases will feature
prominently in this paper:
\begin{enumerate}
\item ($X$ is spin) $y = -1$ and the canonical bundle
  $\cK_X$ admits a square root;\label{it:spin-condition}
\item ($X$ is Calabi--Yau) $y$ is arbitrary and $\cK_X =
  \cO_X$.\label{it:calabi-yau-condition}
\end{enumerate}
The $y=-1$ specialization $\sE_1(X)$ is particularly notable because
it is the historically-earlier notion of elliptic genus due to
Landweber--Stong, Ochanine, and Witten \cite{Landweber1988, Witten1988}.

\subsection{}
\label{sec:wall-crossing-setup}

In this paper, we consider the following geometric setup for studying
(equivariant) wall-crossing problems: $\sT$ is a torus and $M$ is a
{\it $\sT$-equivariant master space}. This means that $M$ is a smooth
proper scheme\footnote{More generally, $M$ can be a smooth
  Deligne--Mumford stack satisfying the weaker properness condition of
  Remark~\ref{rem:master-space-proper}. For elliptic genus, the
  $\sS$-fixed components $Z_\pm$ are still required to be
  schemes. \label{footnote:properness}} with a $(\sT \times
\sS)$-action, where $\sS \coloneqq \bC^\times$ with coordinate denoted
$s$, such that the $\sS$-fixed locus is a disjoint union of the
following $\sT$-invariant pieces:
\begin{enumerate}
\item smooth divisors $\iota_\pm\colon Z_\pm \hookrightarrow M$ with
  normal bundles of $\sS$-weights $s^{\pm 1}$
  respectively;\label{it:master-space-easy-locus}
\item other proper component(s) $\iota_0\colon Z_0 \hookrightarrow M$
  whose normal sheaf $\cN_{\iota_0}$ is locally free;
\end{enumerate}
Generally, in applications, $Z_\pm$ will be two different stable loci
in an ambient algebraic stack.

Many master spaces exist in the literature; examples include
\cite{Thaddeus1996, Okonek1999, Mochizuki2009, Joyce2021,
  Gonzalez2022, Zhou2022}. In many of these, $M$ is only {\it
  quasi-smooth} instead of smooth (see
Remark~\ref{rem:master-space-smooth}).

\subsection{}

\begin{theorem}[Main theorem] \label{thm:invariance}
  Suppose $\cN_{\iota_0}\big|_{Z_0^\sT} = \cE_+ \oplus \cE_-$ only has
  pieces of $\sS$-weight $s^{\pm 1}$, and
  \begin{equation} \label{eq:wall-crossing-rank-condition}
    \rank \cE_+ \equiv \rank \cE_- \bmod{N}
  \end{equation}
  for some integer $N > 0$. Then, for any $N$-th root of
  unity $\zeta_N \neq 1$,
  \[ \sE_{-\zeta_N}(Z_+) = \sE_{-\zeta_N}(Z_-). \]
\end{theorem}

To be clear, $Z_0^\sT$ may have many connected components. While
$\rank \cE_\pm$ need not be constant on $Z_0^\sT$,
\eqref{eq:wall-crossing-rank-condition} must hold on each connected
component.

Note that if $\rank \cE_+ = \rank \cE_-$, i.e. the theorem holds for
all $N > 0$, then $\sE_{-y}(Z_+) = \sE_{-y}(Z_-)$ for general $y$
because $q$-coefficients of $\sE_{-y}$ are Laurent polynomials in $y$.

\subsection{}
\label{sec:rigidity-strategy}

The proof of Theorem~\ref{thm:invariance}, given in
\S\ref{sec:wall-crossing}, is fairly elementary. A very general
wall-crossing procedure, arising from $(\sT \times S)$-equivariant
localization on $M$, expresses the difference between
$\sE_{-y}(Z_+)$ and $\sE_{-y}(Z_-)$ in terms of a contour
integral
\begin{equation} \label{eq:interaction-term-as-contour-integral}
  \oint_{|s|\approx 1} \frac{\Theta(y(\cE_+ + \cE_-))}{\Theta(\cE_+ + \cE_-)} \,\frac{ds}{s}
\end{equation}
on $\sS$. The integrand is a function on $\sS$ in general, but
specializing $y = \zeta_N$ makes it $q^{\cochar\sS}$-invariant, so the
contour integral descends to the elliptic curve $\sS/q^{\cochar \sS}$.
There, the contour encloses all the poles of the integrand, and
therefore vanishes by Cauchy's residue theorem.

\subsection{}

Previous work studying elliptic genus under birational transformations
focused on the Calabi--Yau setting of \ref{it:calabi-yau-condition},
and the {\it non-equivariant} (i.e. $\sT$ is trivial) elliptic genus
\[ \sE_{-y}\colon M\SU_* \otimes \bQ \to \Jac \]
viewed as a homomorphism from the $\SU$-bordism ring to Jacobi forms
in $(y, \tau)$. A bordism argument by Totaro \cite[\S 4]{Totaro2000}
shows that $\sE_{-y}$ is invariant under certain Calabi--Yau flops,
using Krichever--H\"ohn's {\it elliptic rigidity}
\cite{Krichever1990,Hoehn1991}. More sophisticated work by Borisov and
Libgober extends this to arbitrary crepant birational transformations
\cite{Borisov2003}.

It is possible that an equivariant version of such arguments can be
used to prove Theorem~\ref{thm:invariance}, as explained below. But
our approach outlined in \S\ref{sec:rigidity-strategy} is more
versatile; for instance, it applies equally well to virtual chiral
elliptic genus (see \S\ref{sec:virtual-chiral-intro}).

\subsection{}

In fact, a simple geometric formula for the contour integral
\eqref{eq:interaction-term-as-contour-integral} exists in general. In
\S\ref{sec:residue-geometric}, we explain how Jeffrey--Kirwan
integration expresses contour integrals of this shape as
equivariant Euler characteristics on GIT quotients --- in our case
(Proposition~\ref{prop:interaction-residue-as-integral}), as
\begin{equation} \label{eq:interaction-residue-as-euler-characteristic}
  \chi\left(\bP(V), \frac{\Theta(y \cO(1) \otimes V_+ + y^{-1} \cO(1) \otimes V_-)}{\Phi(\cO(1) \otimes V) \Phi(\cO(-1) \otimes V^\vee)}\right)
\end{equation}
where $V \coloneqq V_+ \oplus V_-$ and coordinates of $V_\pm$
correspond to Chern roots of $\cE_\pm$. A slightly different
application of Jeffrey--Kirwan integration expresses
\eqref{eq:interaction-term-as-contour-integral} as the change in
$\sE_{-y}$ across certain toric flips (Remark~\ref{rem:rigidity}). The
invariance of $\sE_{-\zeta_N}$ under these toric flips is equivalent
to our main Theorem~\ref{thm:invariance}.

Unfortunately for wall-crossing, in general
\eqref{eq:interaction-residue-as-euler-characteristic} appears to be a
non-trivial function of the Chern roots of $\cE_\pm$, with no
productive closed form.

\subsection{}

We give two direct applications of Theorem~\ref{thm:invariance}: to
the original Thaddeus master space, studying variation of GIT
(\S\ref{sec:GIT-quotients}), and to Mochizuki's {\it enhanced} master
space, studying moduli of sheaves (\S\ref{sec:moduli-of-sheaves}). In
each, most of the work is to identify sufficiently explicit and useful
criteria such that the topological condition
\eqref{eq:wall-crossing-rank-condition} holds. As such, the following
two theorems describe certain ``nice'' cases which may be less general
than permitted by their constituent pieces.

\subsection{}
\label{sec:applications-GIT}

\begin{theorem}[Theorem~\ref{thm:VGIT-invariance}, Theorem~\ref{thm:VGIT-smooth-exceptional-loci}] \label{thm:elliptic-genus-VGIT}
  Let $X \sslash_{\cL_\pm} \sG$ be two smooth GIT quotients separated
  by a single, simple (\S\ref{sec:VGIT-smooth-assumptions}) wall
  $\cL_0$ in the space of GIT stability conditions, and let
  $X^{\sst}(\cL) \subset X$ denote the $\cL$-semistable locus. The
  natural maps
  \[ (X^{\sst}(\cL_0) \setminus X^{\sst}(\cL_\mp)) \sslash_{\cL_\pm} \sG \to (X^{\sst}(\cL_0) \setminus (X^{\sst}(\cL_+) \cup X^{\sst}(\cL_-))) \sslash_{\cL_0} \sG \]
  are always locally trivial fibrations by weighted projective spaces.
  If in fact they are locally trivial $\bP^{N_\pm}$-fibrations with
  $N_+ - N_- \equiv 0 \bmod{N}$, then
  \[ \sE_{-\zeta_N}(X \sslash_{\cL_+} \sG) = \sE_{-\zeta_N}(X \sslash_{\cL_-} \sG). \]
\end{theorem}

\subsection{}
\label{sec:applications-moduli-of-sheaves}

\begin{theorem}[Theorem~\ref{thm:sheaves-simple-wall}, Lemma~\ref{lem:surface-Ext-parity}, Theorem~\ref{thm:sheaves-wall}, Corollary~\ref{cor:sheaves-wall-spin}] \label{thm:moduli-of-sheaves-invariance}
  Let $Y$ be a smooth projective variety with canonical bundle
  $\cK_Y$. Under the assumptions of \S\ref{sec:moduli-wall-crossing}
  and \S\ref{sec:moduli-smooth-and-proper}, consider two stable loci
  $\fM_\alpha^\sst(\pm) \subset \fM_\alpha^\sst(0)$ in a moduli stack
  of sheaves of class $\alpha \in H^*(Y)$ on $Y$, separated by a
  single wall at $0$.
  \begin{itemize}
  \item (Spin) If $\cK_Y$ admits a square root, then
    \[ \sE_1(\fM_\alpha^\sst(+)) = \sE_1(\fM_\alpha^\sst(-)). \]
  \item (Calabi--Yau) If $\cK_Y^{\otimes k} = \cO_Y$ for some integer
    $k$, and the wall is simple (see \ref{it:simple-wall}), then
    \[ \sE_{-y}(\fM_\alpha^\sst(+)) = \sE_{-y}(\fM_\alpha^\sst(-)). \]
  \end{itemize}
\end{theorem}

Unfortunately, in practice, the
assumption~\ref{it:smooth-moduli-stack} that all semistable loci are
smooth is too strong. For instance, it is satisfied for surfaces $Y$
if $Y$ is Fano (Remark~\ref{rem:moduli-assumptions}), i.e. a del Pezzo
surface, but then $\cK_Y$ does not satisfy either hypothesis of the
theorem.

In the Calabi--Yau case, the assumption that the wall is simple is an
artefact of Mochizuki's setup, and is possibly unnecessary; see
Remark~\ref{rem:auxiliary-obstruction-parity}.

\subsection{}
\label{sec:virtual-chiral-intro}

In \S\ref{sec:virtual-chiral}, we replace the smoothness condition on
$X$ or the master space $M$ with the condition that they admit
equivariantly-symmetric perfect obstruction theories. Then
$\sE_{-y}(X)$ may be replaced by the {\it virtual chiral elliptic
  genus}
\[ \sE^{\vir/2}_{-y}(X) \coloneqq \chi\left(X, \frac{\cO_X^\vir \otimes \cK_\vir^{1/2}}{\Phi(\cT_X^\vir)\Phi((\cT_X^\vir)^\vee)}\right) \]
following ideas of \cite[\S 8]{Fasola2021}, and there is a virtual
version (Theorem~\ref{thm:virtual-invariance}) of our main theorem.
The twist by a square root of the virtual canonical $\cK_\vir
\coloneqq \det(\cT_X^\vir)$ is crucial, as first observed in
\cite{Nekrasov2016}.

Symmetric perfect obstruction theories are a hallmark of
Donaldson--Thomas-type theories on Calabi--Yau $3$-folds, to which we
apply Theorem~\ref{thm:virtual-invariance} as follows.

\subsection{}
\label{sec:donaldson-thomas-intro}

Let $Y$ be a quasi-projective Calabi--Yau $3$-fold, acted on by a
torus with proper fixed loci, such that the (trivial) canonical bundle
has non-trivial weight $y$. For the Donaldson--Thomas (DT) moduli
stack $\fN_\alpha$, and a stability condition $\sigma$ with no
strictly $\sigma$-semistable objects, we define the {\it elliptic DT
  invariant} as the virtual chiral elliptic genus
\[ \DT^{\Ell/2}_{-y}(\alpha; \sigma) \coloneqq \sE^{\vir/2}_{-y}(\fN^\sst_\alpha(\sigma)) \]
of the $\sigma$-stable locus in $\fN_\alpha$, whenever it has proper
torus-fixed loci (Definition~\ref{def:elliptic-DT-invariant}). Note
that the $q=0$ specialization is what is usually referred to as a
K-theoretic DT invariant.

A special case is {\it Vafa--Witten (VW) theory} \cite{Tanaka2020},
where $Y = \tot(\cK_S)$ is an equivariant local surface
(\S\ref{sec:vafa-witten-setup}) for a smooth projective surface $S$
acted on by a torus $\sT$. Up to some modifications to $\fN_\alpha$,
the elliptic DT invariant becomes the {\it elliptic VW invariant}
$\VW^{\Ell/2}_{-y}(\alpha; \sigma)$.

\subsection{}

\begin{theorem}[Theorem~\ref{thm:DT-wall}, Remark~\ref{rem:DT-wall-y}, Lemma~\ref{lem:VW-Ext-parity}, Lemma~\ref{lem:surface-Ext-parity}, Corollary~\ref{cor:VW-wall-spin}, Lemma~\ref{lem:VW-poles}] \label{thm:VW-invariance}
  Consider two stable loci $\fN_\alpha^\sst(\pm) \subset
  \fN_\alpha^\sst(0)$ in the VW moduli stack of class $\alpha \in
  H^*(S)$, separated by a single wall at $0$ (see
  \ref{it:stability-condition}).
  \begin{enumerate}
  \item (Spin) If $\cK_S$ admits a square root, and
    $\cK_S\big|_{S^\sT}$ has non-trivial $\sT$-weight on each
    component, then
    \[ \VW^{\Ell/2}_1(\alpha; +) = \VW^{\Ell/2}_1(\alpha; -). \]
  \item (Calabi--Yau) If $\cK_S^{\otimes k} = \cO_S$ for some
    integer $k$, and the wall is simple (see \ref{it:simple-wall}),
    then
    \[ \VW^{\Ell/2}_{-y}(\alpha; +) = \VW^{\Ell/2}_{-y}(\alpha; -). \]
  \end{enumerate}
\end{theorem}

In contrast to Theorem~\ref{thm:moduli-of-sheaves-invariance}, the
smoothness assumption~\ref{it:smooth-moduli-stack} is no longer
required. For instance, $S$ can be a Hirzebruch surface of even
degree, satisfying the spin condition, or an Enriques surface,
satisfying the Calabi--Yau condition for $k=2$.

As in \S\ref{sec:applications-moduli-of-sheaves}, in the Calabi--Yau
case it may be unnecessary to require the wall to be simple.

\subsection{Notation}

All schemes are Noetherian and over $\bC$. Given a torus $\sT =
(\bC^\times)^r$, its character and cocharacter lattices are
$\Char(\sT) \coloneqq \Hom(\sT, \bC^\times)$ and $\cochar(\sT)
\coloneqq \Hom(\bC^\times, \sT)$ respectively, and:
\begin{itemize}
\item $K_{\sT}(\pt) = \bZ[t^\omega : \omega \in \Char \sT]$ where $t
  \in \sT$ denotes the coordinate;
\item $K_{\sT}(\pt)_{\loc} \coloneqq K_{\sT}(\pt)[(1 - t^\omega)^{-1}
  : 0 \neq \omega \in \Char \sT]$
\item $K_{\sT}(X)_{\loc} \coloneqq K_{\sT}(X) \otimes_{K_{\sT}(\pt)}
  K_{\sT}(\pt)_{\loc}$ is the localized $\sT$-equivariant K-group of
  $X$.
\end{itemize}
The monomials $t^\omega \in K_{\sT}(\pt)$ are referred to as {\it
  $\sT$-weights}.

In K-theory, $\chi(X, -) \coloneqq \sum_i (-1)^i H^i(X, -)$ is the
Euler characteristic, $\wedge_z^\bullet \coloneqq \sum_i z^i \wedge^i$
is the exterior algebra, and all functors are derived, e.g. $\Ext_X(-,
-) \coloneqq \sum_i (-1)^i \Ext_X^i(-, -)$.

Given a closed embedding $\iota\colon Z \hookrightarrow Z'$, let
$\cN_\iota$ or $\cN_{Z \subset Z'}$ denote its normal sheaf.

Finally, $\ln(z) \coloneqq \log(z) / 2\pi i$ so that $\exp(2\pi i \ln
z) = z$.

\subsection{Acknowledgements}

I would like to thank: Nikolas Kuhn, who originally posed to me the
question of computing residues of elliptic classes in wall-crossing;
Andrei Okounkov, who convinced me of the power of contour integrals;
Yehao Zhou, who helped simplify the proof of the main
Theorem~\ref{thm:invariance}.

This work was supported by World Premier International Research Center
Initiative (WPI), MEXT, Japan.

\section{Wall-crossing and proof of the main theorem}
\label{sec:wall-crossing}

\subsection{}

Let $M$ be the $\sT$-equivariant master space from
\S\ref{sec:wall-crossing-setup}, with action by $\sT \times \sS$. Let
$s$ denote the coordinate on $\sS = \bC^\times$.

In this section, we explain the general strategy for obtaining
wall-crossing formulas from $M$, and prove the main
Theorem~\ref{thm:invariance}. It has a long history with many
interesting applications to cohomological and K-theoretic invariants,
particularly when $M$ is allowed to be quasi-smooth instead of smooth.
But it has not yet been systematically applied to elliptic genus.

\subsection{}
\label{sec:master-space-calculation}

Wall-crossing formulas arise from the $(\sT \times
\sS)$-equivariant localization formula on $M$. Explicitly, if
$\cF$ is a coherent sheaf on $M$, then
\begin{equation} \label{eq:master-space-localization}
  \begin{aligned}
    \chi(M, \cF)
    &= \chi\left(Z_-^{\sT}, \frac{1}{\wedge_{-1}^\bullet(\cN_{Z_-^\sT\subset Z_-}^\vee)} \frac{\cF\big|_{Z_-^\sT}}{1 - s \cL_-^\vee}\right)
    +\chi\left(Z_+^{\sT}, \frac{1}{\wedge_{-1}^\bullet(\cN_{Z_+^\sT \subset Z_+}^\vee)} \frac{\cF\big|_{Z_+^\sT}}{1 - s^{-1} \cL_+^\vee}\right) \\
    &\qquad +\chi\left(Z_0^{\sT}, \frac{1}{\wedge_{-1}^\bullet (\cN_{Z_0^\sT\subset Z_0}^\vee)} \frac{\cF|_{Z_0^\sT}}{\wedge_{-1}^\bullet (\cN_{\iota_0}^\vee)}\right),
  \end{aligned}
\end{equation}
where $\cN_{\iota_\pm} \eqqcolon s^{\pm} \cL_\pm$ are the normal line
bundles from \ref{it:master-space-easy-locus}, and several obvious
pullbacks have been omitted. Each integrand is written so that only
the second factor is a non-trivial (rational) function of $s$.

The left hand side $\chi(M, \cF)$ is an element in the {\it
  non-localized} K-group $K_{\sT \times \sS}(\pt)$ since $M$ is
proper and $\cF$ is coherent, but each individual term in the right
hand side of \eqref{eq:master-space-localization} lives in the
localized $K_{\sT \times \sS}(\pt)_{\loc}$.

\subsection{}

The idea is to apply an operation $\res$ on rational functions of
$s$ such that:
\begin{enumerate}
\item $\res\colon K_{\sT \times \sS}(\pt) \mapsto
  0$;\label{it:residue-map-axiom-i}
\item for any $\sT$-equivariant line bundle $\cL$ on a
  Deligne--Mumford stack with trivial
  $\sT$-action, \label{it:residue-map-axiom-ii}
  \[ \res \frac{1}{1 - s\cL} = 1. \]
\end{enumerate}
Property (i) ensures the left hand side of
\eqref{eq:master-space-localization} will vanish, while property (ii)
ensures the first two terms on the right hand side of
\eqref{eq:master-space-localization} will become
\[ \chi\left(Z_\pm^\sT, \frac{\cF\big|_{s^{\mp 1}=\cL_\pm}}{\wedge_{-1}^\bullet(\cN_{Z_\pm^\sT\subset Z_\pm})}\right) = \chi\left(Z_\pm, \cF\big|_{s^{\mp 1}=\cL_\pm}\right) \]
where the equality is $\sT$-equivariant localization on $Z_\pm$.

\subsection{}

As the notation suggests, $\res$ is given by the {\it K-theoretic
  residue map}
\begin{align}
  \res\colon K_{\sT \times \sS}(\pt)_{\loc} &\to K_{\sT}(\pt)_{\loc} \nonumber \\
  f &\mapsto \frac{1}{2\pi i} \oint_{|s| \approx 1} f\, \frac{ds}{s} \label{eq:residue-map}
\end{align}
where the contour\footnote{The notation $|s| \approx 1$ is to suggest
  that, analytically, we treat $\sT$-weights as complex numbers close
  to $1$, but we take $|1 - q| \gg |1 - w|$ for all $\sT$-weights $w$,
  so poles of the form $s^kq^n = w$ for $n \neq 0$ are excluded.}
encloses all poles of the form $s^k = w$ for all $0 \neq k \in \bZ$
and $\sT$-weights $w$. For any Deligne--Mumford stack $Z$ with $(\sT
\times \sS)$-action, we continue to use $\res$ to denote the map
\[ \res\colon K_{\sT \times \sS}(Z)_{\loc} \to K_\sT(Z^{\sS})_{\loc} \]
induced by tensor product with $K(Z^{\sT \times \sS})$. It is
not difficult to check that $\res$ satisfies properties
\ref{it:residue-map-axiom-i} and \ref{it:residue-map-axiom-ii}, and in
fact is uniquely characterized by them.

\subsection{}

\begin{example} \label{ex:residue-simple-pole}
  Let $t$ be any $\sT$-weight. Then $\vartheta(yst)/\vartheta(st)$ has
  a simple pole at $s = t^{-1}$, which is enclosed by $|s| \approx 1$,
  and it has no other poles enclosed by the contour. Hence
  \[ \res \frac{\vartheta(yst)}{\vartheta(st)} = \lim_{s \to t^{-1}} (1 - st) \frac{\vartheta(yst)}{(1 - (st)^{-1}) \phi(st) \phi((st)^{-1})} = -\frac{\vartheta(y)}{\phi(1)^2}. \]
  This still holds when $s$ is replaced by $s^{-1}$, without the minus
  sign on the right hand side.
\end{example}

\subsection{}
\label{sec:elliptic-genus-wcf}

For elliptic genus, take $\cF =
\Theta(y\cT_M)/\Phi(\cT_M)\Phi(\cT_M^\vee)$ in
\eqref{eq:master-space-localization} and apply $\res$. Recall that in
K-theory, there are splittings such as
\[ \cT_M\big|_{Z_\pm^\sT} = s^{\pm} \cL_\pm \big|_{Z_\pm^\sT} + \cN_{Z_\pm^\sT \subset Z_\pm} + \cT_{Z_\pm^\sT}, \]
and $\Theta$ and $\Phi$ are multiplicative. Combined with the
calculation in Example~\ref{ex:residue-simple-pole}, we get
\begin{equation} \label{eq:elliptic-genus-WCF}
  0 = \frac{\vartheta(y)}{\phi(1)^2} \left(\sE_{-y}(Z_-) - \sE_{-y}(Z_+)\right) + \chi\left(Z_0^\sT, \cdots \otimes \res \frac{\Theta(y \cN_{\iota_0})}{\Theta(\cN_{\iota_0})}\right),
\end{equation}
a wall-crossing formula relating $\sE_{-y}(Z_-)$ and $\sE_{-y}(Z_+)$.
In the remaining integrand, the terms represented by $\cdots$ are
irrelevant because our goal is to study the case where the residue
vanishes, and therefore $\sE_{-y}(Z_-) = \sE_{-y}(Z_+)$.

The assumption $y \neq 1$ of the main Theorem~\ref{thm:invariance} is
important for this step (and only this step), in order for the right
hand side of \eqref{eq:elliptic-genus-WCF} to be non-trivial.

\subsection{}

\begin{remark} \label{rem:master-space-smooth}
  Smoothness of $M$ is crucial for \eqref{eq:elliptic-genus-WCF} to
  hold. It is natural to ask whether we can allow $M$ to be only {\it
    quasi-} or {\it virtually} smooth, meaning that $M$ may not be
  smooth but instead has a perfect obstruction theory
  \cite{Behrend1997}. The naive answer is no (cf.
  \S\ref{sec:virtual-chiral}), for the following reason.

  For quasi-smooth schemes, the elliptic genus $\sE_{-y}$ should be
  upgraded to the {\it virtual elliptic genus} $\sE_{-y}^\vir$ by
  replacing all instances of $\cT$ in \eqref{eq:elliptic-genus} with
  the virtual tangent bundle $\cT^\vir$ \cite{Fantechi2010}. Virtual
  localization \cite{Graber1999} provides a virtual version of
  \eqref{eq:master-space-localization}. However, the residue arguments
  fail to work for $\sE_{-y}^\vir$: the left hand side of
  \eqref{eq:elliptic-genus-WCF} is no longer zero because the term
  \[ \Theta(y\cT_M^\vir) \in K_{\sT \times \sS}(M)_{\loc} \]
  in the integrand $\cF =
  \Theta(y\cT_M^\vir)/\Phi(\cT_M^\vir)\Phi((\cT_M^\vir)^\vee)$ may now
  have poles at $|s| \approx 1$. Put differently, $\chi(M, \cF)$ is
  now an element of $K_{\sT \times \sS}(\pt)_{\loc}$, not $K_{\sT
    \times \sS}(\pt)$, and there is no good way to control its poles
  in $s$. So virtual elliptic genus is not amenable to our
  wall-crossing setup.
\end{remark}

\subsection{}

\begin{remark} \label{rem:master-space-proper}
  Properness of $M$ is not crucial for \eqref{eq:elliptic-genus-WCF}
  to hold. We used properness to conclude that $\res \chi(M, \cF) =
  0$, but it can be replaced by the weaker property that
  \begin{itemize}
  \item for any $(\sT \times \sS)$-weight $w$ with non-trivial
    $\sS$-component, the fixed locus $M^{\sT_w}$ is proper, where
    $\sT_w \subset \ker(w)$ is the maximal torus.
  \end{itemize}
  A priori, $\chi(M, \cF) \in K_{\sT \times \sS}(\pt)_{\loc}$, but by
  applying the pole cancellation Lemma~\ref{lem:pole-cancellation}
  below, it has no poles at $w = 1$ for any $w$ with non-trivial
  $\sS$-component. In other words,
  \[ \chi(M, \cF) \in K_{\sT}(\pt)_{\loc} \otimes_\bZ K_{\sS}(\pt). \]
  This is enough to imply $\res \chi(M, \cF) = 0$.
\end{remark}

\subsection{}

\begin{lemma}[Pole cancellation] \label{lem:pole-cancellation}
  Let $M$ be a scheme acted on by a torus $\sT$. Let $w$ be a
  $\sT$-weight and $\sT_w \subset \ker(w)$ be the maximal torus. If
  the $\sT_w$-fixed locus of $M$ is proper, then
  \[ \chi(M, \cF)\big|_{w=1} \in K_{\sT_w}(\pt)_{\loc} \]
  is well-defined. In particular $\chi(M, \cF)$ has no pole at $w=1$.
\end{lemma}

\begin{proof}
  This is a geometric observation from \cite[Proposition
    3.2]{Arbesfeld2021}, see also \cite[Lemma 5.5]{Liu2023}.
  Properness of $M^{\sT_w}$ means that $\sT_w$-equivariant
  localization can be used to compute the right hand side of
  \[ \chi(M, \cF)\big|_{w=1} = \chi\left(M, \cF\big|_{w=1}\right). \]
  The result is an element of $K_{\sT_w}(\pt)_{\loc}$. 
\end{proof}

\subsection{}

\begin{remark}
  The abstract study of residue maps, meaning module homomorphisms
  satisfying \ref{it:residue-map-axiom-i} and
  \ref{it:residue-map-axiom-ii}, was first explicitly done in
  \cite{Metzler2002}. There, \eqref{eq:residue-map} appeared in the
  slightly different form
  \[ f \mapsto -\left(\Res_{s=0} + \Res_{s=\infty}\right)\left(f \, \frac{ds}{s}\right) \]
  where $\Res$ denotes the usual notion of the residue of a
  differential $1$-form. Since elements $f \in K_{\sT \times
    \sS}(\pt)_{\loc}$ have poles, in $s$, only in the set $S \coloneqq
  \{0, \infty\} \sqcup \{|s| \approx 1\}$, this form is equivalent to
  the original by the residue theorem. However, this is no longer true
  for elliptic functions $f$, which, in addition to poles in $S$, have
  poles at every $q$-shift of $S$.
\end{remark}

\subsection{}

Using the wall-crossing formula \eqref{eq:elliptic-genus-WCF}, the
main Theorem~\ref{thm:invariance} is reduced to the following.

\begin{proposition} \label{prop:residue-vanishing}
  Let $\cE_\pm$ be equivariant vector bundles of $\sS$-weight $s^{\pm
    1}$, respectively, on a scheme with trivial $(\sT \times
  \sS)$-action. If
  \[ \rank \cE_+ \equiv \rank \cE_- \bmod{N} \]
  for some integer $N > 0$, then, for any $N$-th root of unity $\zeta_N$,
  \begin{equation} \label{eq:interaction-residue-vanishing}
    \left(\res \frac{\Theta(y(\cE_+ + \cE_-))}{\Theta(\cE_+ + \cE_-)}\right)\bigg|_{y=\zeta_N} = 0.
  \end{equation}
\end{proposition}

In \S\ref{sec:virtual-rigidity}, we will explain why this theorem
continues to hold even if $\cE_\pm$ are virtual bundles. On the other
hand, the assumption that $\cE_\pm$ have $\sS$-weights $\pm 1$,
respectively, is crucial and cannot be removed.

\subsection{}
\label{sec:interaction-residue-chern-roots}

\begin{proof}[Proof of Proposition~\ref{prop:residue-vanishing}.]
  To be very explicit, the left hand side of
  \eqref{eq:interaction-residue-vanishing}, before specializing to $y
  = \zeta_N$, is
  \begin{equation} \label{eq:interaction-residue-1}
    \oint_{|s| \approx 1} \prod_i \frac{\vartheta(ysa_i, \tau)}{\vartheta(sa_i, \tau)} \prod_j \frac{\vartheta(ys^{-1}b_j, \tau)}{\vartheta(s^{-1}b_j, \tau)} \,\frac{ds}{s}
  \end{equation}
  for variables $\{a_i\}_i$ and $\{b_j\}_j$ which will eventually be
  specialized to the $\sT$-equivariant Chern roots of $\cE_+$ and
  $\cE_-$ respectively. These Chern roots have the form $w \otimes
  \cL$, where $w$ is a $\sT$-weight and $\cL$ is a non-equivariant
  line bundle. In particular, $\cL$ is unipotent:
  \[ (1 - \cL)^{\otimes M} = 0, \qquad \forall M \gg 0. \]
  Analytically, we can therefore treat the variable corresponding to
  the Chern root the same way as the $\sT$-weight $w$.

\subsection{}

  Let $I(s)$ denote the integrand in \eqref{eq:interaction-residue-1},
  so that $I(s)$ is a meromorphic $1$-form on $\sS$ once the
  $\sT$-weights are fixed to be some appropriately-generic elements in
  $\sS$. Using the basic $q$-difference equation $\vartheta(qs) =
  -(qs)^{-1} \vartheta(s)$, one checks easily that
  \[ I(qs) = y^{\rank \cE_- - \rank \cE_+} I(s). \]
  Hence, specializing $y = \zeta_N$, both the integrand $I(s)$ and the
  contour integral \eqref{eq:interaction-residue-1} descend to the
  elliptic curve $E_\sS \coloneqq \sS/q^{\cochar\sS}$. But on $E_\sS$,
  all poles of $I(s)$ are enclosed by the contour, and therefore the
  contour integral is zero by Cauchy's residue theorem and properness
  of $E_\sS$.
\end{proof}

This concludes the proof of the main Theorem~\ref{thm:invariance} as
well. \qed

\section{A geometric formula for the residue}
\label{sec:residue-geometric}

\subsection{}

In this section, we study the general contour integral
\[ C_n(\vec a, \tau; y) \coloneqq \oint_{|s|\approx 1} s^n \prod_{i=1}^{r_+} \frac{\vartheta(ysa_i; \tau)}{\vartheta(sa_i; \tau)} \prod_{j=1}^{r_-} \frac{\vartheta(y^{-1} sb_j; \tau)}{\vartheta(sb_j; \tau)} \, \frac{ds}{s}, \]
where the $\vec a \coloneqq (a_1, \ldots, a_{r_+}, b_1, \ldots,
b_{r_-})$ are viewed as fixed, generic elements of $\sS$ with $|a_i|,
|b_j| \gg 1$. Using that $\vartheta(z^{-1}; \tau) = -z \vartheta(z;
\tau)$, clearly $C_0(\vec a, \tau; y)$ is the contour integral
\eqref{eq:interaction-residue-1} up to an overall factor $y^{r_-}$.

The main results are a geometric formula
(Proposition~\ref{prop:interaction-residue-as-integral}) for $C_n(\vec
a, \tau; y)$, and some mild control over where $C_n(\vec a, \tau; y)$
has poles, as a meromorphic function of $\vec a$ and $\tau$. Neither
result is critical for any other part of the paper and may be safely
skipped on a first reading.

We assume throughout this section that $r_+ + r_- > 0$.

\subsection{}

\begin{proposition} \label{prop:interaction-residue-as-integral}
  Let $\sA \coloneqq (\bC^\times)^{r_+ + r_-}$ with coordinates
  identified with $\vec a$. Then
  \begin{equation} \label{eq:interaction-residue-as-integral}
    C_n(\vec a, \tau; y) = \chi\left(\bP(V), \cO(n) \otimes \frac{\Theta(y\cO(1) \otimes V_+ + y^{-1}\cO(1) \otimes V_-)}{\Phi(\cO(1) \otimes V) \Phi(\cO(-1) \otimes V^\vee)}\right) \in K_{\sA}(\pt)[y^{\pm 1}]\pseries*{q}
  \end{equation}
  is an $\sA$-equivariant Euler characteristic, where $V \coloneqq V_+
  \oplus V_- \coloneqq \bC^{r_+} \oplus \bC^{r_-}$ and $\sA \subset
  \GL(V)$ acts as the maximal torus.
\end{proposition}

Note that $\bP(V)$ is proper, and each coefficient of $y$ and $q$ is
an element of $K_\sA(\bP(V))$, so indeed the result is valued in the
Laurent polynomial ring $K_\sA(\pt)$ instead of its localization
$K_\sA(\pt)_\loc$.

\subsection{}

\begin{remark}
  While this paper mostly focuses on the case where $C_0(\vec a, \tau;
  y)$ vanishes for some specialization of $y$, it is possible that
  this explicit formula may be useful for the wall-crossing in
  \S\ref{sec:wall-crossing}. However, it is clear from this formula
  that $C_0(\vec a, \tau; y)$ will depend non-trivially on the
  coordinates $a_i \in \sA$. In wall-crossing, these correspond to the
  Chern roots of $\cN_{\iota_0}$, and typically one has very little
  control over these Chern roots.
\end{remark}

\subsection{}

\begin{proof}[Proof of Proposition~\ref{prop:interaction-residue-as-integral}.]
  This is a standard application of the Jeffrey--Kirwan residue
  formula for integrals over GIT quotients; see \cite[Appendix
    A]{Aganagic2018}, or the more general \cite[Proposition
    2.4]{Szenes2004} (written in cohomology, not K-theory).

  To summarize the basic idea in our setting, let $\cF$ be a coherent
  sheaf on $\bP(V) = (V \setminus 0) / \sS$ which is induced by
  restriction from an $\sS$-equivariant coherent sheaf $\tilde\cF$ on
  $V$. Concretely, $\cF$ is a coefficient of $y$ and $q$ in the
  integrand of \eqref{eq:interaction-residue-as-integral}. Then
  \[ \chi((V \setminus 0) / \sS, \cF(m)) = \chi(V \setminus 0, \tilde \cF \otimes s^m)^{\sS} \cong \chi(V, \tilde\cF \otimes s^m)^{\sS}. \]
  Analytically, $\chi(V, \tilde\cF \otimes s^m)$ converges to a
  rational function on $\sA \times \sS$ for $|a_i| \gg 1$, with no
  poles on the maximal compact subgroup $\{|s| = 1\} \subset \sS$, so
  \begin{align*}
    \chi(V, \tilde\cF \otimes s^m)^\sS
    &= \int_{|s|=1} \chi(V, \tilde\cF \otimes s^m) \, \frac{ds}{s} \\
    &= \int_{|s|=1} \frac{\tilde\cF\big|_0 \otimes s^m}{\prod_i (1 - a_i^{-1} s^{-1}) \prod_j (1 - b_j^{-1} s^{-1})} \, \frac{ds}{s}
  \end{align*}
  where the second equality is $(\sA \times \sS)$-equivariant
  localization on $V$. Since $\tilde\cF|_0$ has bounded degree in $s$,
  for $m \gg 0$ there is no pole at $s=0$ and therefore $\{|s|=1\}$
  and $\{|s|\approx 1\}$ enclose the same poles. We obtain the desired
  formula upon specializing $m = 0$. This is valid because
  \[ \chi((V \setminus 0) / \sS, \cF(m)) \in \bQ[m, x_1^{\pm m}, x_2^{\pm m}, \ldots], \]
  where the $x_i$ may be roots of unity or $(\sA \times \sS)$-weights,
  and for elements of such a ring, if an equality holds for all $m \gg
  0$, then in fact it holds for all $m \in \bZ$.
\end{proof}

\subsection{}

\begin{remark} \label{rem:rigidity}
  By the same reasoning as in the proof of
  Proposition~\ref{prop:interaction-residue-as-integral}, it turns out
  that the vanishing of $C_0(\vec a, \tau; \zeta_N)$
  (Proposition~\ref{prop:residue-vanishing}) is equivalent to the
  invariance of elliptic genus under certain toric flips. Namely, the
  contour integral in \eqref{eq:interaction-residue-1} may also be
  expressed as
  \begin{equation} \label{eq:rigidity-as-toric-flip}
    \sE_{-y}(\tot(\cO_{\bP(V_+)}(-1) \otimes V_-^\vee)) - \sE_{-y}(\tot(\cO_{\bP(V_-)}(-1) \otimes V_+^\vee))
  \end{equation}
  where $\tot$ denotes total space. These non-compact toric geometries
  $Y_\pm \coloneqq \tot(\cO_{\bP(V_\pm)}(-1) \otimes V_\mp^\vee)$ are
  related by the birational transformation
  \begin{equation} \label{eq:Ppm-flip}
    \begin{tikzcd}
      Y_+ \ar{dr}[swap]{f_+} && Y_- \ar{dl}{f_-} \\
      & Y_0
    \end{tikzcd}
  \end{equation}
  where $f_\pm$ contracts the zero section $\bP(V_\pm)$. If
  $\pi_\pm\colon Y_\pm \to \bP(V_\pm)$ denotes the projection, then
  the canonical bundles are
  \begin{align*}
    \cK_{Y_\pm}
    &= \pi_\pm^*\left(\cK_{\bP(V_\pm)} \otimes \det (\cO_{\bP(V_\pm)}(1) \otimes V_\mp)\right) \\
    &= \pi_\pm^*\cO_{\bP(V_\pm)}(\dim V_\pm - \dim V_\mp) \otimes \det(V_\pm) \det(V_\mp),
  \end{align*}
  so \eqref{eq:Ppm-flip} is a flip in general and a flop if and only
  if $\dim V_+ = \dim V_-$. For instance, the classical Atiyah flop is
  modeled by the case $\dim V_+ = \dim V_- = 2$.

  In particular, if $\dim V_- = 0$, then the vanishing of
  \eqref{eq:rigidity-as-toric-flip} becomes $\sE_{-\zeta_N}(\bP^{N-1})
  = 0$.
\end{remark}

\subsection{}
\label{sec:interaction-residue-poles}

We turn to studying poles of $C_n(\vec a, \tau; y)$ as a meromorphic
function on $\sA \times \bH$, where $\bH \ni \tau$ is the upper half
plane. Recall that $\vartheta(z; \tau)$ is a holomorphic function of
$\tau$. By construction, poles of $C_n(\vec a, \tau; y)$ can only
occur at
\[ \Poles(q^m) \coloneqq \bigcup_\mu\, \{q^m \vec a^\mu = 1\} \subset \sA \times \bH \]
for various $m \in \bZ$, where $\mu$ ranges over finitely many
non-trivial $\sA$-characters.\footnote{More precisely, by
$\sA$-equivariant localization applied to
\eqref{eq:interaction-residue-as-integral}, $\vec a^\mu$ ranges over
all $\sA$-weights of the normal bundle of $\bP(V)^\sA \subset
\bP(V)$.} Proposition~\ref{prop:interaction-residue-as-integral}
allows us to be more precise about where poles actually occur.

\subsection{}

\begin{proposition}\label{prop:interaction-residue-real-holomorphicity}
  $C_n(\vec a, \tau; y)$ is holomorphic on $\Poles(q^0)$.
\end{proposition}

\begin{proof}
  This is the same argument as in \cite[Proposition 5.1]{Bott1989} or
  \cite[Lemma 1.3]{Liu1996}, so we give only a sketch. The idea is to
  obtain more mileage from
  Proposition~\ref{prop:interaction-residue-as-integral} by computing
  \eqref{eq:interaction-residue-as-integral} via $\sA$-equivariant
  localization on $\bP(V)$. The resulting expression, in the domain
  \begin{equation} \label{eq:localization-domain}
    \bigcap_\mu\, \{|q| < |\vec a^\mu| < |q|^{-1}\}, \quad 0 < |q| < 1,
  \end{equation}
  has an expansion of the form $\sum_{j \ge 0} c_j(\vec a) q^j$ where
  the $c_j(\vec a)$ are rational functions of $\vec a$ which only have
  poles on $\Poles(\sA; q^0)$, and the possible locations of such
  poles are independent of $j$. On the other hand, from properness, we
  also know that the $c_j(\vec a)$ are Laurent polynomials of bounded
  degree (polynomial in $j$), and in particular have no poles on
  $\Poles(q^0)$. This is enough to conclude the desired holomorphicity
  since \eqref{eq:localization-domain} contains $\Poles(q^0)$.
\end{proof}

\section{Example: GIT quotients}
\label{sec:GIT-quotients}

\subsection{}

Let $X$ be a smooth projective variety with the action of a reductive
group $\sG$. Pick two ample linearizations $\cL_\pm$ which lie in
adjacent chambers in the space of ample linearizations of this
$\sG$-action \cite[\S 3.3]{Dolgachev1998} \cite[\S 2]{Thaddeus1996}.
Explicitly, let
\[ \cL_t \coloneqq \cL_+^{(1+t)/2} \otimes \cL_-^{(1-t)/2}, \qquad t \in [-1, 1], \]
and assume, without loss of generality, that $\cL_0$ is the only point
that is ``on the wall'', i.e. there are strictly $\cL_t$-semistable
points only for $t = 0$. Let $X^{\sst}(t)$ denote the
$\cL_t$-semistable locus in $X$, and write
\[ X \sslash_t \sG \coloneqq X^{\sst}(t) / \sG \]
for the GIT quotient, for short. We assume $X \sslash_\pm \sG$ are
smooth, and refer to this setup as {\it variation of GIT} ({\it
  VGIT}).

If $X$ is acted on by a torus $\sT$ commuting with $\sG$, then clearly
everything above can be made $\sT$-equivariant.

\subsection{}

Thaddeus \cite[\S 3]{Thaddeus1996} constructs a master space for VGIT
as follows: for the natural $\sG$-action on $\cL_+ \oplus \cL_-$, it
is the projective GIT quotient
\[ M \coloneqq \bP(\cL_+ \oplus \cL_-) \sslash_{\cO(1)} \sG, \]
where the ample line bundle $\cO(1)$ has the canonical linearization.
The torus $\sS \coloneqq \bC^\times$, with coordinate $s$, acts on $M$
by scaling the $\cL_+$ factor with weight $s$. Let
\[ \iota_\pm\colon X \cong \bP(\cL_\pm) \hookrightarrow \bP(\cL_+ \oplus \cL_-) \]
be the inclusion of the $0$ and $\infty$ sections, and
$\bP^\circ(\cL_+ \oplus \cL_-) \coloneqq \bP(\cL_+ \oplus \cL_-)
\setminus (\im(\iota_+) \cup \im(\iota_-))$.

\subsection{}
\label{sec:VGIT-smooth-assumptions}

Throughout this section, we assume that $\cL_0$ is a {\it simple
  wall}, meaning that for all $x$ in
\[ X^0 \coloneqq X^{\sst}(0) \setminus (X^{\sst}(+) \cup X^{\sst}(-)): \]
\begin{enumerate}
\item the stabilizer $\sG_x$ of the $\sG$-action on $x$ is
  $\bC^\times$;
\item $\sG_x$ acts on the fiber $(\cL_+ \otimes \cL_-^\vee)|_x$ by
  scaling with weight one.
\end{enumerate}

\begin{lemma} \label{lem:VGIT-master-space}
  Under these assumptions, $M$ is a $\sT$-equivariant master space in
  the sense of \S\ref{sec:wall-crossing-setup}.
\end{lemma}

\begin{proof}[Proof sketch.]
  We must check $M$ is a smooth scheme. The semistable locus in
  $\bP(\cL_+ \oplus \cL_-)$ decomposes as
  \begin{equation} \label{eq:VGIT-master-space-semistable-locus}
    \iota_+(X^{\sst}(+)) \cup \iota_-(X^{\sst}(-)) \cup \bigcup_{t \in [-1,1]} \pi^{-1}(X^{\sst}(t)) \cap \bP^\circ(\cL_+ \oplus \cL_-)
  \end{equation}
  where $\pi\colon \bP(\cL_+ \oplus \cL_-) \to X$ is the projection.
  Points in the image of the first two terms in $M$ must be smooth
  because $X \sslash_\pm \sG$ are smooth. The assumptions imply the
  $\sG$-action on the third term is free. See \cite[\S
    4]{Thaddeus1996} for details.
\end{proof}

Note that $(\cL_+ \otimes \cL_-^\vee)|_x \setminus \{0\}$ is the fiber
of the projection $\bP^\circ(\cL_+ \oplus \cL_-) \to X$.

\subsection{}

\begin{remark}[Toric varieties]
  When $\sG$ is a torus, the assumptions of
  \S\ref{sec:VGIT-smooth-assumptions} automatically hold, possibly
  after replacing $\cL_+$ and $\cL_-$ by some positive powers. In
  particular, if $X = \bC^r$ and $\sG \subset (\bC^\times)^r$ is a
  sub-torus of the maximal torus of $\GL(r)$, then $X \sslash_\pm \sG$
  are toric varieties and VGIT reduces to the combinatorics of lattice
  polyhedra.

  Recall that, conversely, all toric varieties without torus factor
  have the form $\bC^r \sslash_\theta \sG$, where $\sG \subset
  (\bC^\times)^r$ is an algebraic subgroup and $\theta$ is a
  $\sG$-weight.
\end{remark}

\subsection{}
\label{sec:VGIT-master-space-fixed-loci}

In the decomposition
\eqref{eq:VGIT-master-space-semistable-locus}, clearly the first
two terms $\iota_\pm(X^{\sst}(\pm))$ are $\sS$-fixed and induce
inclusions which we also denote
\[ \iota_\pm\colon X \sslash_\pm \sG \hookrightarrow M, \]
whose normal bundles are $\cL_\pm$. In the third term, a point becomes
$\sS$-fixed in $M$ if and only if it has positive-dimensional
stabilizer under the $(\sG \times \sS)$-action. By hypothesis, such
points must belong to $\pi^{-1}(X^0)$, so the remaining $\sS$-fixed
locus is exactly
\[ \iota_0\colon (\pi^{-1}(X^0) \cap \bP^\circ(\cL_+ \oplus \cL_-)) \sslash_{\cO(1)} \sG \hookrightarrow M \]
and the stabilizer $\sG_x$ of points in $X^0$ is identified with
$\sS$. Thus $\cN_{\iota_0}$ is identified with $\cN_{X^0 \subset X}$.

\subsection{}

\begin{theorem} \label{thm:VGIT-invariance}
  Under the assumptions of \S\ref{sec:VGIT-smooth-assumptions}, if
  $\cN_{X^0 \subset X}\big|_{(X^0)^\sT} = \cE_+ \oplus \cE_-$ only has
  pieces of $\sS$-weight $s^{\pm 1}$, and
  \[ \rank \cE_+ \equiv \rank \cE_- \bmod{N} \]
  for some integer $N > 0$, then $\sE_{-\zeta_N}(X \sslash_+ \sG) =
  \sE_{-\zeta_N}(X \sslash_- \sG)$.
\end{theorem}

This is a restatement of the main Theorem~\ref{thm:invariance} in the
current setting, using the description of $\sS$-fixed loci in
\S\ref{sec:VGIT-master-space-fixed-loci}. It remains to describe the
$\sS$-weight pieces of $\cN_{X^0 \subset X}$ in some useful way, which
Thaddeus provides and we summarize as
Theorem~\ref{thm:VGIT-smooth-exceptional-loci} below.

\subsection{}

\begin{theorem}[{\cite[Proposition 4.6, Theorem 4.7]{Thaddeus1996}}] \label{thm:VGIT-smooth-exceptional-loci}
  Let $X^\pm \coloneqq X^{\sst}(0) \setminus X^{\sst}(\mp)$. Under the
  assumptions of \S\ref{sec:VGIT-smooth-assumptions},
  \[ X^\pm \sslash_\pm \sG \to X^0 \sslash_0 \sG \]
  are locally trivial fibrations whose fibers are weighted projective
  spaces $\bP(|w_i^\pm|)$, where $\{w_i^\pm\}_i \in \bZ$ are (the
  exponents of) the positive and negative $\sS$-weights, respectively,
  of $\cN_{X^0 \subset X}$.
\end{theorem}

\begin{proof}[Proof sketch.]
  This claim may be checked affine-locally. For example, the
  Bialynicki-Birula decomposition theorem shows that the positive and
  negative $\sS$-weight parts of $\cN_{X^0 \subset X}$ are equal to
  $\cN_{X^0 \subset X^\pm}$, respectively.
\end{proof}



\subsection{}

\begin{example}[Blow-ups] \label{ex:VGIT-blowup-invariance}
  Let $Y \coloneqq X \sslash_- \sG$ and suppose that $X \sslash_+ \sG
  = \Bl_p Y$ is the blow-up of a point $p \in Y$. (Conversely, a large
  class of blow-ups of points in GIT quotients can be realized as VGIT
  \cite{Hu2000}.) Let $N \coloneqq \dim Y - 1$ and assume that the
  wall is simple. Then $X^\pm \sslash_\pm \sG \to X^0 \sslash_0 \sG$
  are the exceptional loci, which are $\bP^N$ and $\bP^0$-fibrations
  respectively, so Theorem~\ref{thm:VGIT-invariance} says
  \begin{equation} \label{eq:VGIT-blowup-invariance}
    \sE_{-y}(Y) = \sE_{-y}(\Bl_p Y) \;\; \text{ if } y = \zeta_N.
  \end{equation}

  Note that if $Y$ is toric and $p$ is a torus-fixed point, then,
  using torus-equivariant localization as in
  \eqref{eq:elliptic-genus-localized}, the equality of elliptic genera
  is equivalent to the identity
  \begin{equation} \label{eq:VGIT-blowup-invariance-toric}
    \prod_{i=1}^{N+1} \frac{\vartheta(yx_i)}{\vartheta(x_i)} = \sum_{i=1}^{N+1} \frac{\vartheta(yx_i)}{\vartheta(x_i)} \prod_{j \neq i} \frac{\vartheta(yx_j/x_i)}{\vartheta(x_j/x_i)} \;\; \text{ if } y = \zeta_N.
  \end{equation}
  Here, $x_i$ are weights of $T_p Y$. In the blow-up, $p$ is replaced
  by $N+1 = \dim Y$ different torus-fixed points, corresponding to the
  terms in the sum on the right hand side.
\end{example}

\subsection{}

\begin{remark}
  In fact, some toric computation can strengthen the results in
  Example~\ref{ex:VGIT-blowup-invariance}. In particular we claim that
  \eqref{eq:VGIT-blowup-invariance} (as a claim about a class of
  varieties $Y$), and therefore
  \eqref{eq:VGIT-blowup-invariance-toric}, is an ``if and only if''.
  First, by explicit computation,
  \begin{equation} \label{eq:elliptic-genus-PN}
    \sE_{-y}(\bP^{N+1}) = \chi\left(\bP^{N+1}, \wedge_{-y}^\bullet(\cT_{\bP^{N+1}}^\vee)\right) + \cdots = \frac{1 - y^{N+2}}{1 - y} + \cdots
  \end{equation}
  where $\cdots$ denotes terms involving $q$. Second, if $\pi\colon A
  \to B$ is a $\bP^n$-fibration, then
  \[ \sE_{-y}(A) = \chi\left(B, \frac{\Theta(y \cT_B)}{\Phi(\cT_B)\Phi(\cT_B^\vee)} \otimes \pi_* \frac{\Theta(y \cT_\pi)}{\Phi(\cT_\pi) \Phi(\cT_\pi^\vee)}\right) = \sE_{-y}(\bP^n) \sE_{-y}(B). \]
  The first equality is the projection formula for $\pi$, and the
  second equality is because all relative cohomology bundles
  $R^i\pi_*(\cT_\pi^j)$ are canonically trivialized by powers of the
  hyperplane class. Hence, for a torus-fixed point $p \in \bP^{N+1}$,
  the $\bP^1$-fibration $\pi\colon \Bl_p \bP^{N+1} \to \bP^N$ implies
  \[ \sE_{-y}(\Bl_p \bP^{N+1}) = \sE_{-y}(\bP^1) \sE_{-y}(\bP^N) = \frac{1 - y^2}{1 - y} \frac{1 - y^{N+1}}{1 - y} + \cdots. \]
  Comparing with \eqref{eq:elliptic-genus-PN}, the $q$-constant terms
  match if and only if $y = \zeta_N$. So the hypothesis in
  \eqref{eq:VGIT-blowup-invariance} is really necessary.
\end{remark}

\section{Example: moduli of sheaves}
\label{sec:moduli-of-sheaves}

\subsection{}
\label{sec:moduli-wall-crossing}

Let $Y$ be a smooth projective variety acted on by a torus $\sT$, and
$\cat{A} \subset D^b\cat{Coh}(Y)$ be an abelian subcategory of its
bounded derived category of coherent sheaves.\footnote{For the
  purposes of this section, with some care it is also possible to
  consider an exact subcategory $\cat{B} \subset \cat{A}$ which may
  not be abelian but is closed under isomorphisms and direct sum, see
  e.g. the setup of \cite[\S 5.1]{Joyce2021}.} We refer to $\cat{A}$
as a {\it moduli of sheaves} even though the objects involved may in
general be complexes of sheaves.

Given $\alpha \in H^*(Y)$, a typical wall-crossing problem in
$\cat{A}$ involves a continuous family $\{\sigma_\xi\}_{\xi \in
  [-1,1]}$ of stability conditions\footnote{We use Joyce's notion of
  stability condition: functions $\sigma$ from non-zero classes
  $\alpha$ into some totally-ordered set, such that if $\alpha = \beta
  + \gamma$ for $\alpha,\beta,\gamma \neq 0$, then either
  $\sigma(\beta) > \sigma(\alpha) > \sigma(\gamma)$ or $\sigma(\beta)
  = \sigma(\alpha) = \sigma(\gamma)$ or $\sigma(\beta) <
  \sigma(\alpha) < \sigma(\gamma)$.} such that:
\begin{enumerate}
\item\label{it:stability-condition} for $\xi \neq 0$, there are no
  strictly $\sigma_\xi$-semistable objects $E \in \cat{A}$ with $\ch(E)
  = \alpha$;
\item\label{it:moduli-stack} for any $\xi$, there exist algebraic
  moduli stacks (acted on by $\sT$)
  \[ \fM^\sst_\beta(\xi) \coloneqq \{E \in \cat{A} : E \text{ is } \sigma_\xi\text{-semistable and } \ch(E) = \beta\} \]
  for all relevant $\beta$, which includes $\alpha$
  and the classes appearing in
  \eqref{eq:strictly-semistable-sheaf-splitting}.
\end{enumerate}
The goal is to relate elliptic genus of $\fM_\alpha^\sst(+)$ and
$\fM_\alpha^\sst(-)$, where $\pm$ means $\pm 1$. Continuity implies
that all $\{\sigma_\xi\}_{\xi \in (0,1]}$ are equivalent, and similarly
  for $\{\sigma_\xi\}_{\xi \in [-1, 0)}$.

\subsection{}
\label{sec:moduli-smooth-and-proper}

To consider enumerative invariants, one usually makes a properness
assumption:
\begin{enumerate}
\item\label{it:proper-moduli-stack} the moduli stacks
  $\fM^\sst_\alpha(\pm)$ are proper algebraic spaces.\footnote{For
  this paper, the words ``algebraic space'' and ``scheme'' are
  basically interchangeable. In practice $\fM_\alpha^\sst(\pm)$ are
  often projective schemes.}
\end{enumerate}
To be precise, all objects in $\cat{A}$ have at least a $\bC^\times$'s
worth of scalar automorphisms, and $\fM^\sst_\alpha(\xi)$ denotes the
{\it rigidified} moduli stacks where this $\bC^\times$ has been
removed \cite[Appendix A]{Abramovich2008}. So, objects in
$\fM^\sst_\alpha(\pm)$ have no non-trivial automorphisms and
$\fM^\sst_\alpha(\pm)$ are automatically algebraic spaces.

Our main Theorem~\ref{thm:invariance} requires the master space to be
smooth; see Remark~\ref{rem:master-space-smooth}. To satisfy this, it
is typically enough to assume that:
\begin{enumerate}[resume]
\item\label{it:smooth-moduli-stack} for any $\xi$, in particular $\xi
  = 0$, the moduli stack $\fM^\sst_\alpha(\xi)$ is smooth.
\end{enumerate}
In particular $\fM^\sst_\alpha(\pm)$ are smooth and proper, so their
elliptic genera $\sE_y(\fM^\sst_\alpha(\pm))$ exist.

\subsection{}

\begin{remark} \label{rem:moduli-assumptions}
  Conditions~\ref{it:stability-condition}, \ref{it:moduli-stack} and
  \ref{it:proper-moduli-stack} are relatively weak for most
  wall-crossing problems of interest. For instance, if for any $\xi$,
  \[ \Ext_Y^{< 0}(E, E) = 0, \quad \forall \, [E] \in \fM^\sst_\beta(\xi), \]
  then \ref{it:moduli-stack} holds \cite{Lieblich2006}. There also
  exists machinery \cite{Alper2023} for verifying
  \ref{it:proper-moduli-stack} in great generality, or, more
  concretely, one can often use Langton's strategy for semistable
  reduction \cite{Langton1975}.

  On the other hand, condition~\ref{it:smooth-moduli-stack} is very
  strong --- it is basically the requirement that for any $\xi$,
  including $\xi = 0$ where there exist strictly semistable objects,
  \begin{equation} \label{eq:smooth-moduli-stack}
    \Ext_Y^{> 1}(E, E) = 0, \quad \forall \, [E] \in \fM^\sst_\alpha(\xi).
  \end{equation}
  This is automatic if $\dim Y \le 1$. If $\dim Y = 2$, this is
  equivalent to the vanishing of $\Ext_Y^2(E, E) = \Hom_Y(E, E \otimes
  \cK_Y)^\vee$ which, for example, holds if $\cK_Y$ is anti-ample by a
  standard degree argument for semistable objects \cite[Proposition
    1.2.6]{Huybrechts2010}, cf. \cite[Definition 7.47]{Joyce2021}. If
  $\dim Y \ge 3$, this is typically hopeless.
\end{remark}

\subsection{}
\label{sec:simple-wall}

We begin with the case of a {\it simple wall} (cf.
\S\ref{sec:VGIT-smooth-assumptions}), namely:
\begin{enumerate}
\item\label{it:simple-wall} all strictly semistable $[E] \in
  \fM^\sst_\alpha(0)$ split as $E = E_1 \oplus E_2$ where $E_1, E_2$
  are both $\sigma_0$-stable.
\end{enumerate}
In other words, the automorphism group of objects in
$\fM_\alpha^\sst(0)$ is at worst $\bC^\times$, given by scaling $E_1$.
Then, the strategy behind the Thaddeus master space can be applied
directly to obtain a master space; this is done explicitly in the
general wall-crossing machinery of Mochizuki \cite[\S 1.3, \S
  1.6.1]{Mochizuki2009} or of Joyce \cite{Joyce2021}.\footnote{Both
  \cite{Mochizuki2009} and \cite{Joyce2021} are written in a vastly
  more general setting where the
  assumptions~\ref{it:stability-condition},
  \ref{it:proper-moduli-stack} and \ref{it:smooth-moduli-stack} may
  not hold. They must pass to auxiliary moduli stacks, impose a
  quasi-smoothness assumption and work with virtual cycles, like in
  \S\ref{sec:virtual-chiral} and \S\ref{sec:donaldson-thomas}. In our
  simpler setting, these complications may be ignored.}

In the notation of \S\ref{sec:wall-crossing-setup}, the complicated
locus $Z_0$ in the master space is
\begin{equation} \label{eq:strictly-semistable-sheaf-splitting}
  Z_0 = \bigsqcup_{\substack{\alpha_1+\alpha_2=\alpha\\\sigma_0(\alpha_1)=\sigma_0(\alpha_2)}} \left\{[E_1 \oplus E_2] \in \fM_\alpha^\sst(0) : [E_i] \in \fM_{\alpha_i}^\sst(0) \text{ stable} \right\} \hookrightarrow \fM_\alpha^\sst(0).
\end{equation}

\subsection{}

\begin{remark}
  If $\alpha$ is rank-$2$ and torsion-free, then all walls are simple
  because the rank can only decompose as $2 = 1 + 1$. For rank greater
  than two, typically non-simple walls exist.
\end{remark}

\subsection{}

\begin{theorem}[Simple wall] \label{thm:sheaves-simple-wall}
  Assume \ref{it:simple-wall}. If, for all $[E_1 \oplus E_2] \in
  Z_0^\sT$,
  \[ \dim \Ext_Y(E_1, E_2) \equiv \dim \Ext_Y(E_2, E_1) \bmod{N} \]
  for some integer $N > 0$, then $\sE_{-\zeta_N}(\fM_\alpha^\sst(+)) =
  \sE_{-\zeta_N}(\fM_\alpha^\sst(-))$.
\end{theorem}

Lemma~\ref{lem:surface-Ext-parity} below gives some situations in
which the dimension condition is satisfied.

\begin{proof}
  The $\sS$-action scales $E_1$ with weight $s$, so the normal bundle
  is given by
  \begin{equation} \label{eq:interaction-locus-normal-bundle}
    \cN_{\iota_0}\Big|_{[E_1 \oplus E_2]} = -s^{-1} \Ext_Y(E_1, E_2) - s\Ext_Y(E_2, E_1).
  \end{equation}
  We conclude by a direct application of Theorem~\ref{thm:invariance}.
\end{proof}

\subsection{}

\begin{lemma} \label{lem:surface-Ext-parity}
  Let $S$ be a smooth projective surface with canonical bundle
  $\cK_S$, and take $E, F \in \cat{A}$.
  \begin{enumerate}
  \item\label{lem:surface-Ext-parity:spin} If $\cK_S$ admits a square
    root, then $\dim \Ext_S(E, F) \equiv \dim \Ext_S(F, E) \bmod{2}$.
  \item\label{lem:surface-Ext-parity:CY} If $\cK_S^{\otimes k} =
    \cO_S$ for some integer $k$, then $\dim \Ext_S(E, F) = \dim
    \Ext_S(F, E)$.
  \end{enumerate}
\end{lemma}

\begin{proof}
  By Serre duality, we are comparing $\dim \Ext_S(E, F)$ and $\dim
  \Ext_S(E, F \otimes \cK_S)$. Using bilinearity of $\Ext_S$, assume
  that $E, F \in \cat{Coh}(S)$. Also, assume $E$ is locally free; if
  not, take a locally free resolution and consider each term of the
  resolution. So, without loss of generality, we are comparing $\dim
  \chi(S, F)$ and $\dim \chi(S, F \otimes \cK_S)$.
  Hirzebruch--Riemann--Roch says
  \[ \dim \chi(S, F) = \rank(\cF)\chi(S) - \frac{1}{2} c_1(F) K + \ch_2(F) \]
  where $K \coloneqq c_1(\cK_S)$ is the canonical divisor. Then
  \[ \dim \chi(S, F) - \dim \chi(S, F \otimes \cK_S) = \frac{1 - \rank(F)}{2} K^2 - c_1(F) K. \]
  For (i), $K = 2D$ for some $D$, so this is divisible by $2$. For
  (ii), $K = 0$ so this is zero.
\end{proof}

\subsection{}

\begin{remark}
  From Remark~\ref{rem:moduli-assumptions}, if we only consider the
  case of Fano surfaces $Y$, which are blow-ups of $\bP^2$ at $\le 8$
  points, then the canonical divisor
  \[ K_Y = \pi^* K_{\bP^2} + \sum_i E_i \]
  is a primitive vector. Here $E_i$ denote the exceptional divisors.
  So the hypotheses of Lemma~\ref{lem:surface-Ext-parity} cannot be
  satisfied for Fano surfaces $Y$, and indeed it is unclear if either
  of them can ever be satisfied at all when $\dim Y > 1$. Nonetheless,
  the results of this section will be useful in
  \S\ref{sec:donaldson-thomas}.
\end{remark}

\subsection{}
\label{sec:mochizuki-auxiliary-category}

Generally, walls in $\cat{A}$ are not simple, meaning that
condition~\ref{it:simple-wall} does not hold. In \cite{Mochizuki2009}
and \cite{Joyce2021}, this issue is solved by lifting the
wall-crossing problem to the auxiliary abelian category
\[ \tilde{\cat{A}}^{\Fr} \coloneqq \{(E, V^\bullet) : E \in \cat{A} \text{ and } V^\bullet \text{ is a full flag in}\, \Fr(E)\} \]
associated to a {\it framing functor}, which is an exact functor
\begin{equation} \label{eq:framing-functor}
  \Fr\colon \cat{A}' \to \cat{Vect}
\end{equation}
on a full exact subcategory $\cat{A}' \subset \cat{A}$ containing all
objects of interest, such that
\[ \Hom(E, E) \to \Hom(\Fr(E), \Fr(E)) \]
is injective for all $E \in \cat{A}'$.\footnote{The injectivity
  condition is explicitly stated by Joyce \cite[Assumption
    5.1(g)(iii)]{Joyce2021} but not mentioned by Mochizuki, who only
  uses $\Fr(E) \coloneqq H^0(E \otimes \cO_Y(m))$.} For example, if
$\cat{A} \subset \cat{Coh}(Y)$, a common choice is $\Fr(E) \coloneqq
H^0(E \otimes \cO_Y(m))$ for $m \gg 0$. See
\S\ref{sec:flag-combinatorics} for our notion of a full flag.

The additional data of the flag ``resolves'' the non-simple wall in
$\cat{A}$ into multiple simple walls in $\tilde{\cat{A}}^{\Fr}$. Each
such simple wall can be crossed like in \S\ref{sec:simple-wall}; the
so-called {\it enhanced master space}, for the auxiliary stacks, sits
inside a flag variety fibration over the original master space and is
therefore still smooth. The complicated locus $Z_0^\sT$ in the master
space now involves splittings
\begin{equation} \label{eq:full-flag-splitting}
  (E, V^\bullet) = (E_1 \oplus E_2, V_1^\bullet \oplus V_2^\bullet), \quad \sigma_0(E_1) = \sigma_0(E_2)
\end{equation}
where $V_i^\bullet$ is a full flag in $\Fr(E_i)$ and each $(E_i,
V_i^\bullet)$ is $\sigma_0$-stable.

\subsection{}

\begin{theorem} \label{thm:sheaves-wall}
  If, for all splittings \eqref{eq:full-flag-splitting} appearing in all
  auxiliary wall-crossings,
  \[ \dim \Ext_Q(V_1^\bullet, V_2^\bullet) - \dim \Ext_Y(E_1, E_2) \equiv \dim \Ext_Q(V_2^\bullet, V_1^\bullet) - \dim \Ext_Y(E_2, E_1) \bmod{N} \]
  for some integer $N > 0$, then $\sE_{-\zeta_N}(\fM_\alpha^\sst(+)) =
  \sE_{-\zeta_N}(\fM_\alpha^\sst(-))$.
\end{theorem}

Note that this is not as widely applicable as
Theorem~\ref{thm:sheaves-simple-wall}, because we get very little
control over the $\Ext_Q$ terms; see
Lemma~\ref{lem:interaction-locus-full-flag-parity} below.

\begin{proof}
  The normal bundle is now, cf.
  \eqref{eq:interaction-locus-normal-bundle},
  \begin{equation} \label{eq:interaction-locus-enhanced-normal-bundle}
    \begin{split}
      \cN_{\iota_0}\Big|_{[(E_1 \oplus E_2, V_1^\bullet \oplus V_2^\bullet)]}
      &= s^{-1}\left(\Ext_Q(V_1^\bullet, V_2^\bullet) - \Ext_Y(E_1, E_2)\right) \\[-.5em]
      &\quad +s\left(\Ext_Q(V_2^\bullet, V_1^\bullet) - \Ext_Y(E_2, E_1)\right).
    \end{split}
  \end{equation}
  Here, viewing flags as representations of a type A quiver $Q$, the
  quiver part of the deformation theory is given by the standard
  formula
  \begin{equation} \label{eq:quiver-Ext}
    \Ext_Q(V_1^\bullet, V_2^\bullet) \coloneqq \sum_i \left(\Hom(V_1^i, V_2^{i+1}) - \Hom(V_1^i, V_2^i)\right).
  \end{equation}
  We conclude by a direct application of Theorem~\ref{thm:invariance}.
\end{proof}

\subsection{}
\label{sec:flag-combinatorics}

Let $W$ be a vector space. For us, $V^\bullet$ being a {\it full flag}
of length $K$ in $W$ means that
\[ \dim V^k \le \dim V^{k+1} \le \dim V^k + 1 \]
for $0 \le k \le K$, with the convention that $V^0 = 0$ and $V^{K+1} =
W$. We write $\dim V^\bullet \coloneqq \dim W$.

\begin{lemma} \label{lem:interaction-locus-full-flag-parity}
  Let $\Split(V^\bullet; d_1, d_2)$ be the set of splittings
  $V^\bullet = V_1^\bullet \oplus V_2^\bullet$ of a full flag into two
  smaller full flags with $\dim V_i^\bullet = d_i$. Then
  \[ \left\{\dim \Ext_Q(V_1^\bullet, V_2^\bullet) - \dim \Ext_Q(V_2^\bullet, V_1^\bullet)\right\} = \{-d_1d_2, -d_1d_2+2, \ldots, d_1d_2-2, d_1d_2\} \]
  where the left hand side ranges over all splittings in
  $\Split(V^\bullet; d_1, d_2)$.
\end{lemma}

\begin{proof}
  Note that if $V_i^k = V_i^{k+1}$ for $i = 1,2$ and some $k$, then by
  \eqref{eq:quiver-Ext} we can remove the $k$-th step from both flags
  without affecting $\dim \Ext_Q(V_1^\bullet, V_2^\bullet)$ or $\dim
  \Ext_Q(V_2^\bullet, V_1^\bullet)$. So, without loss of generality,
  $V^\bullet$ has the shortest possible length $d \coloneqq \dim
  V^\bullet$. Then it is convenient to use the bijection
  \begin{align*}
    \Split(V^\bullet; d_1, d_2) &\xrightarrow{\sim} \{I \subset \{1,2,\ldots,d\} : |I| = d_1\} \\
    (V_1^\bullet, V_2^\bullet) &\mapsto \{i : V_1^i \neq V_1^{i+1}\}.
  \end{align*}
  Let $v_i^k \coloneqq \dim V_i^k$, and write the quantity of interest
  as $\ext_Q^1(I, +) - \ext_Q^1(I, -)$ where
  \[ \ext^1_Q(I, +) \coloneqq \sum_i v_1^i v_2^{i+1}, \quad \ext^1_Q(I, -) \coloneqq \sum_i v_2^iv_1^{i+1}. \]
  It is straightforward that, for $I_{\max} \coloneqq \{1, 2, \ldots,
  d_1\}$ and $I_{\min} = \{d_1+1, d_1+2, \ldots, d\}$,
  \begin{alignat*}{3}
    \ext^1_Q(I_{\text{max}}, +) &= d_1d_2,\quad &\ext^1_Q(I_{\max}, -) &= 0 \\
    \ext^1_Q(I_{\text{min}}, +) &= 0,\quad &\ext^1_Q(I_{\min}, -) &= d_1d_2.
  \end{alignat*}
  These result in the maximum and minimum values $\pm d_1d_2$. For the
  intermediate values, suppose $I$ and $I'$ differ by replacing an
  element $k$ by $k+1$. This amounts to $(v_1')^{k+1} = v_1^{k+1} -
  1$, and consequently $(v_2')^{k+1} = v_2^{k+1} + 1$, while all other
  dimensions remain unchanged, so
  \[ \ext^1_Q(I', +) = \ext^1_Q(I, +) + v^k_1 - v_2^{k+2}, \quad \ext^1_Q(I', -) = \ext^1_Q(I, -) + v_1^{k+2} - v_2^k. \]
  Since $|I| = |I'| = d_1$, it must be that $v_i^{k+2} - v_i^k = 1$
  for $i = 1, 2$.
\end{proof}

\subsection{}

\begin{corollary} \label{cor:sheaves-wall-spin}
  If $Y$ is a smooth projective surface whose canonical bundle admits
  a square root, then the hypothesis of Theorem~\ref{thm:sheaves-wall}
  is satisfied for $N = 2$.
\end{corollary}

Lemma~\ref{lem:interaction-locus-full-flag-parity} shows that there is
no way to generalize this to $N > 2$, in contrast to e.g.
Theorem~\ref{thm:sheaves-simple-wall} in the setting of
Lemma~\ref{lem:surface-Ext-parity:CY}.

\begin{proof}
  By Lemma~\ref{lem:interaction-locus-full-flag-parity}, it suffices
  to ensure that $\dim(V_1^\bullet) \dim(V_2^\bullet) = \dim \Fr(E_1)
  \dim \Fr(E_2)$ is always even. This can be done by replacing the
  framing functor $\Fr$ in the wall-crossing machinery with
  $\Fr^{\oplus 2}$.
\end{proof}

\subsection{}

\begin{remark} \label{rem:auxiliary-obstruction-parity}
  We expect that the $N=2$ restriction of
  Corollary~\ref{cor:sheaves-wall-spin} is an artefact of the choice
  of auxiliary category $\tilde{\cat{A}}^{\Fr}$, rather than an
  intrinsic limitation. Namely, for any given $N \ge 2$, one may
  speculate that there exist auxiliary categories $\tilde{\cat{A}}_N$
  which work equally well for wall-crossing, for which the
  contribution from the ``auxiliary'' part of the obstruction theory,
  i.e. the dimensions in
  Lemma~\ref{lem:interaction-locus-full-flag-parity}, are all $0
  \bmod{N}$ instead of merely $0 \bmod{2}$.

  For example, the Calabi--Yau case of
  Theorem~\ref{thm:moduli-of-sheaves-invariance} should hold {\it
    without} the assumption that the wall is simple. Evidence for this
  includes the fact that if $Y$ is a K3 or abelian surface, $\alpha$
  is primitive (and not too small), and $\{\sigma_\xi\}_\xi$ is a
  general family of Gieseker stability conditions, then
  $\fM_\alpha^\sst(+)$ is deformation-equivalent to
  $\fM_\alpha^\sst(-)$ \cite[Theorems 0.1 and 8.1]{Yoshioka2001} and
  therefore their elliptic genera are equal.
\end{remark}

\section{The virtual chiral version}
\label{sec:virtual-chiral}

\subsection{}

Let $X$ be a proper scheme, and, instead of assuming $X$ is smooth,
assume the weaker condition:
\begin{enumerate}
\item\label{it:symmetric-POT} $X$ has a $\tilde\sT$-equivariant
  perfect obstruction theory \cite{Behrend1997}, obtained from a
  perfect complex\footnote{One typically assumes that $\bE$ admits a
    two-term resolution by vector bundles. However, recent technical
    advances \cite[\S 5]{Aranha2022} suggest that this condition is
    unnecessary, so we do not worry about it.} $\bE \in
  D^b\cat{Coh}_{\sT}(X)$ which satisfies
  \begin{equation} \label{eq:symmetric-POT}
    \bE = y \otimes \bE^\vee[1]
  \end{equation}
  for some non-trivial $\tilde\sT$-weight $y$; we say $\bE$ is {\it
    equivariantly symmetric}.
\end{enumerate}
Then $X$ has a virtual structure sheaf $\cO^\vir_X$ and virtual
tangent bundle $\cT_X^\vir$ \cite{Kontsevich1995,
  Ciocan-Fontanine2009}, both elements of $K_{\tilde\sT}(X)$. Let
$\cK_\vir \coloneqq \det(\cT_X^\vir)^\vee$ be the virtual canonical.

The notation $\tilde\sT$, instead of $\sT$, is a reminder that $y$ is
now an equivariant weight as opposed to a formal variable. For
instance, if $X$ is smooth with $\sT$-action, then let $\bC^\times$
scale fibers of the cotangent bundle $Y \coloneqq T^*X$ with weight
$y$ and view $X$ as its zero section. Then $X$ satisfies
\ref{it:symmetric-POT} with $\tilde\sT \coloneqq \sT \times
\bC^\times$ and $\bE = (\cT_Y^\vee - y \cT_Y)|_X$ in K-theory.

\subsection{}

Whenever a square root $\cK_\vir^{1/2}$ exists, following \cite[\S
  3.1]{Nekrasov2016} let
\[ \hat\cO_X^\vir \coloneqq \cO_X^\vir \otimes \cK_\vir^{1/2} \in K_{\tilde\sT}(X) \]
After possibly passing to a double cover of $\tilde\sT$ so that
$y^{1/2}$ exists, the {\it virtual chiral elliptic genus} of $X$ is
\[ \sE^{\vir/2}_{-y}(X) \coloneqq \chi\left(X, \frac{\hat\cO_X^\vir}{\Phi(\cT_X^\vir)\Phi((\cT_X^{\vir})^\vee)}\right) \in K_{\tilde\sT}(\pt)\pseries*{q}, \]
cf. \cite{Fasola2021}. The expression $\sE^{\vir/2}_{-y_0}(X)$, for
$y_0 \in \bC^\times$, means the specialization of
$\sE^{\vir/2}_{-y}(X)$ to $y = y_0$. Since $X$ is proper, the only
poles in $y$ are at $0$ and $\infty$, so this specialization is always
well-defined.

The deformation invariance of virtual cycles immediately implies the
deformation invariance of $\sE^{\vir/2}_{-y}(X)$.

\subsection{}

\begin{remark}
  In \cite[\S 8.1]{Fasola2021}, virtual chiral elliptic genus was
  defined without requiring the equivariant symmetry
  \eqref{eq:symmetric-POT}, and without inserting $\cK_\vir^{1/2}$. In
  that generality, it has no hope of being a truly elliptic class (see
  \S\ref{sec:elliptic-genus-is-elliptic}) and will not have any of the
  nice wall-crossing properties considered in this paper.
\end{remark}

\subsection{}

\begin{remark} \label{rem:APOTs}
  The perfect obstruction theory in condition~\ref{it:symmetric-POT}
  is really only used to construct $\cO^\vir_X$ and $\cT^\vir_X$.
  There are many weaker, more local notions that also suffice, with
  obvious analogues of the equivariant symmetry
  \eqref{eq:symmetric-POT}:
  \begin{enumerate}
  \item weak perfect obstruction theories (equivalent to complex
    Kuranishi structures) \cite{Oh2024};
  \item almost perfect obstruction theories \cite{Kiem2020a};
  \item semi-perfect obstruction theories \cite{Chang2011}.
  \end{enumerate}
  These are ordered such that (i)$\implies$(ii)$\implies$(iii), and it
  is known that virtual localization holds at least for (ii)
  \cite{Kiem2020}. The content of this section therefore holds at the
  level of (ii).
\end{remark}

\subsection{}
\label{sec:virtual-master-space}

For wall-crossing with virtual chiral elliptic genus, we assume the
$\tilde\sT$-equivariant master space $M$ is a proper
scheme\footnote{Like in footnote~\ref{footnote:properness}, more
  generally $M$ can be a Deligne--Mumford stack satisfying the weaker
  properness condition of Remark~\ref{rem:master-space-proper}.}
satisfying \ref{it:symmetric-POT} for an action of $\tilde\sT \times
\sS$ where $\sS = \bC^\times$, and the $\sS$-fixed locus is a disjoint
union of the following $\tilde\sT$-invariant pieces (cf.
\S\ref{sec:wall-crossing-setup}):
\begin{enumerate}
\item $\iota_\pm\colon Z_\pm \hookrightarrow M$ with
  $\cN_{\iota_\pm}^\vir = \cL_\pm - y^{-1} \cL^\vee$ for line bundles
  $\cL_\pm$ of $\sS$-weights $s^{\pm 1}$;
\item other proper component(s) $\iota_0\colon Z_0 \hookrightarrow M$
  with
  \[ \cN_{\iota_0}^\vir = \cN_0^{\vir/2} - y^{-1} (\cN_0^{\vir/2})^\vee \]
  for some virtual bundle $\cN_0^{\vir/2}$.
\end{enumerate}
Here $\cN_f^\vir$ denotes the {\it virtual} normal bundle of the
closed embedding $f$, namely the $\sS$-moving part of the restriction
$f^*\cT_M^\vir$.

To emphasize, unlike in the non-virtual setting, $M$ does not need to
be actually smooth (cf. Remark~\ref{rem:master-space-smooth}).

\subsection{}

\begin{theorem}[Virtual analogue of Theorem~\ref{thm:invariance}] \label{thm:virtual-invariance}
  Suppose $\cN_0^{\vir/2}\big|_{Z_0^\sT} = \cE_+ \oplus \cE_-$ only
  has pieces of $\sS$-weight $s^{\pm 1}$, and
  \begin{equation} \label{eq:virtual-wall-crossing-rank-condition}
    \rank \cE_+ \equiv \rank \cE_- \bmod{N}
  \end{equation}
  for some integer $N > 0$. Then, for any $N$-th root of unity
  $\zeta_N \neq 1$,
  \[ \sE^{\vir/2}_{-\zeta_N}(Z_+) = \sE^{\vir/2}_{-\zeta_N}(Z_-). \]
\end{theorem}

Here $\cE_\pm$ may be {\it virtual} vector bundles, i.e. of the form
$\cE_\pm^1 - \cE_\pm^2$ for genuine vector bundles $\cE_\pm^1$ and
$\cE_\pm^2$.

\subsection{}

\begin{proof}[Proof of Theorem~\ref{thm:virtual-invariance}.]
Since all steps are analogous to the proof of
Theorem~\ref{thm:invariance}, we only indicate what needs to be
modified.
  
All wall-crossing considerations from \S\ref{sec:wall-crossing}
continue to hold, with the following minor adjustments:
\begin{itemize}
\item in \S\ref{sec:master-space-calculation}, the {\it virtual}
  localization formula \cite{Graber1999} is used to obtain
  \eqref{eq:master-space-localization};
\item in \S\ref{sec:elliptic-genus-wcf}, the integrand is $\cF =
  \hat\cO_M^\vir/\Phi(\cT_M^\vir)\Phi((\cT_M^\vir)^\vee)$;
\item in \S\ref{sec:elliptic-genus-wcf}, the resulting wall-crossing
  formula \eqref{eq:elliptic-genus-WCF} is
  \[ 0 = \frac{y^{\frac{1}{2}}\vartheta(y)}{\phi(1)^2} \left(\sE^{\vir/2}_{-y}(Z_-) - \sE^{\vir/2}_{-y}(Z_+)\right) + \chi\left(Z_0^\sT, \cdots \otimes \res \frac{(\det \cN_{\iota_0}^\vir)^{-\frac{1}{2}}}{\Theta(\cN_{\iota_0}^\vir)}\right). \]
\end{itemize}
To compare with \eqref{eq:elliptic-genus-WCF} more closely, note that
\[ \frac{(\det \cN_{\iota_0}^\vir)^{-\frac{1}{2}}}{\Theta(\cN_{\iota_0}^\vir)} = y^{\frac{1}{2} \rank \cN_0^{\vir/2}}\frac{\Theta(y\cN_0^{\vir/2})}{\Theta(\cN_0^{\vir/2})}. \]

\subsection{}
\label{sec:virtual-rigidity}

It remains to explain why Theorem~\ref{prop:residue-vanishing}
continues to hold for when $\cE_\pm$ are allowed to be virtual vector
bundles. If $\{a_i\}_i$ and $\{b_j^{-1}\}$ are the Chern roots of
$\cE_+^1$ and $\cE_-^1$ respectively, like in
\S\ref{sec:interaction-residue-chern-roots}, and $\{a_k'\}_k$ and
$\{(b_l')^{-1}\}_l$ are the Chern roots of $\cE_+^2$ and $\cE_-^2$
respectively, then the contour integral of interest is
\[ \oint_{|s| \approx 1} \prod_i \frac{\vartheta(ysa_i, \tau)}{\vartheta(sa_i, \tau)} \prod_j \frac{\vartheta(ys^{-1}b_j^{-1}, \tau)}{\vartheta(s^{-1}b_j^{-1}, \tau)} \prod_k \frac{\vartheta(sa_k', \tau)}{\vartheta(ysa_k', \tau)} \prod_l \frac{\vartheta(s^{-1}(b_l')^{-1}, \tau)}{\vartheta(ys^{-1}(b_l')^{-1}, \tau)} \,\frac{ds}{s}. \]
Upon the change of variables $a_k'' \coloneqq y a_k'$ and $b_l''
\coloneqq y b_l'$, the result is exactly the original contour integral
\eqref{eq:interaction-residue-1} using the variables $\{a_i\}_i \cup
\{b_l''\}_l$ and $\{(a_k'')^{-1}\}_k \cup \{b_j^{-1}\}_j$. Note that
$|a_k'| = |a_k''|$ and $|b_l'| = |b_l''|$ since $y$ is eventually
specialized to a root of unity, so, analytically, the new variables
$\{a_k''\}_k$ and $\{b_l''\}_l$ may be treated the same way as the old
variables $\{a_i\}_i$ and $\{b_j\}_j$. Hence the remainder of the
proof of Theorem~\ref{prop:residue-vanishing} can proceed in exactly
the same way.

This concludes the proof of Theorem~\ref{thm:virtual-invariance}.
\end{proof}

\section{Example: Donaldson--Thomas theory}
\label{sec:donaldson-thomas}

\subsection{}

Let $Y$ be a quasi-projective Calabi--Yau $3$-fold acted on by a torus
$\tilde\sT$ such that
\[ \cK_Y = y \otimes \cO_Y \]
for a $\tilde\sT$-weight $y$, and the $\tilde\sT$-fixed locus in $Y$
is proper. For instance, $Y$ could be toric and $\tilde\sT =
(\bC^\times)^3$ the standard torus, or, more generally, $Y$ can be a
local curve or surface.

We work in the setup of \S\ref{sec:moduli-wall-crossing}, denoting the
stability condition by $\sigma$, but using the moduli substack
\begin{equation} \label{eq:DT-fixed-determinant-moduli-stack}
  \fN_\alpha^\sst(\sigma) \coloneqq \{\det E = \cL\} \subset \fM_\alpha^\sst(\sigma)
\end{equation}
of fixed-determinant objects. Here the line bundle $\cL$ must be
chosen such that $c_1(\cL)$ agrees with the $H^2(Y)$ component of
$\alpha \in H^*(Y)$. Such moduli stacks $\fN$ are part of the {\it
  Donaldson--Thomas (DT) theory} of $Y$, whose characterizing property
is the existence of an equivariantly-symmetric perfect obstruction
theory given at the point $[E]$ by
\begin{equation} \label{eq:DT-obstruction-theory}
  \bE^\vee[-1]\Big|_{[E]} = \Ext_Y(E, E)_0,
\end{equation}
where the subscript $0$ denotes trace-less part \cite[Theorem
  3.30]{Thomas2000}. We will assume that $\rank \alpha > 0$ until
\S\ref{sec:vafa-witten-setup}.

\subsection{}

Let $\sigma$ be a stability condition with no strictly semistable
objects in class $\alpha$. The moduli stack $\fN_\alpha^\sst(\sigma)$
is an algebraic space (see \S\ref{sec:moduli-smooth-and-proper}), but
is generally not proper since $Y$ is not proper. So, throughout this
section, we assume the following analogue of
\ref{it:proper-moduli-stack}:
\begin{enumerate}
\item \label{it:DT-moduli-stack-Ty-proper} the $\sT_y$-fixed
  locus in $\fN_\alpha^\sst(\sigma)$ is proper, where $\sT_y
  \subset \ker(y)$ is the maximal torus.
\end{enumerate}
Since $\sT_y \subset \tilde\sT$, this implies:
\begin{enumerate}[resume]
\item \label{it:DT-moduli-stack-T-proper} the $\tilde\sT$-fixed locus
  in $\fN_\alpha^\sst(\sigma)$ is proper.
\end{enumerate}
But \ref{it:DT-moduli-stack-Ty-proper} is a stronger assumption, and
we really need its full strength to study wall-crossing for elliptic
DT invariants (Definition~\ref{def:elliptic-DT-invariant}).

In practice, properness of the $\sT_y$-fixed locus in $Y$ is
usually enough to imply \ref{it:DT-moduli-stack-Ty-proper}. This is a
much more manageable condition. For instance, if $Y$ is toric, it is
equivalent to the condition that non-compact edges in its toric
$1$-skeleton cannot have $\tilde\sT$-weight $y^k$ for any $k \in \bZ$
(but weights $y^k w$ for non-trivial $w$ are allowed).

\subsection{}

\begin{definition}\label{def:elliptic-DT-invariant}
  Let $\alpha \in H^*(Y)$ and $\sigma$ be a stability condition with
  no strictly semistable objects in $\cat{A}$ of class $\alpha$.
  Assume~\ref{it:DT-moduli-stack-T-proper}. Then the {\it elliptic DT
    invariant} is
  \begin{equation} \label{eq:elliptic-DT-invariant}
    \DT_{-y}^{\Ell/2}(\alpha; \sigma) \coloneqq \sE^{\vir/2}_{-y}(\fN_\alpha^\sst(\sigma)) \in K_{\tilde\sT}(\pt)_{\loc}\pseries*{q},
  \end{equation}
  where the virtual chiral elliptic genus is defined by
  $\tilde\sT$-equivariant localization, i.e. as the right hand side of
  \eqref{eq:elliptic-genus-localized-general}. Consequently the result
  lives in the {\it localized} K-group, and in particular it may have
  non-trivial poles in $\{|y| = 1\}$. But if we further assume
  \ref{it:DT-moduli-stack-Ty-proper}, then no such poles in $y$ exist
  by Lemma~\ref{lem:pole-cancellation}, and $y$ may be specialized to
  any root of unity.
\end{definition}

\subsection{}

\begin{remark} \label{rem:reduced-obstruction}
  In \eqref{eq:DT-fixed-determinant-moduli-stack},
  $\fN_\alpha^{\sst}(\sigma)$ has obstruction theory given by the
  traceless $\Ext_Y(E, E)_0$ while $\fM_\alpha^{\sst}(\sigma)$ has
  obstruction theory given by $\Ext_Y(E, E)$. If $H^1(\cO_Y) \neq 0$,
  the latter contains trivial summands $H^1(\cO_Y)$ and its dual
  $H^2(\cO_Y) = y^{-1} H^1(\cO_Y)^\vee$, and therefore all enumerative
  invariants of $\fM_\alpha^{\sst}(\sigma)$ vanish. This is why we
  generally work with $\fN_\alpha^{\sst}(\sigma)$.

  However, the discrepancy between $\Ext$ and traceless $\Ext$ does
  not affect the ``off-diagonal'' terms $\Ext_Y(E_1, E_2)$ and
  $\Ext_Y(E_2, E_1)$ appearing in arguments below. This is true of
  many reductions one may want to perform on the obstruction theory,
  e.g. \cite[\S 1.6]{Tanaka2020} in Vafa--Witten theory.
\end{remark}

\subsection{}

\begin{theorem} \label{thm:DT-wall}
  Assume \ref{it:DT-moduli-stack-Ty-proper} and use the notation of
  \S\ref{sec:moduli-of-sheaves}.
  \begin{enumerate}
  \item (Simple wall) Assume \ref{it:simple-wall}. If, for all $[E_1
    \oplus E_2] \in Z_0^{\tilde\sT}$,
    \[ \dim \Ext_Y(E_1, E_2) \equiv 0 \bmod{N} \]
    for some integer $N > 0$, then $\DT_{-\zeta_N}^{\Ell/2}(\alpha; +) =
    \DT_{-\zeta_N}^{\Ell/2}(\alpha; -)$.
    
  \item\label{thm:DT-wall:it:2} (General wall) If, for all splittings
    \eqref{eq:full-flag-splitting} appearing in all auxiliary
    wall-crossings,
    \[ \dim \Ext_Y(E_1, E_2) \equiv \dim \Ext_Q(V_1^\bullet, V_2^\bullet) - \dim \Ext_Q(V_2^\bullet, V_1^\bullet) \bmod{N} \]
    for some integer $N > 0$, then
    $\DT_{-\zeta_N}^{\Ell/2}(\alpha; +) = \DT_{-\zeta_N}^{\Ell/2}(\alpha; -)$.
  \end{enumerate}
\end{theorem}

This is the direct analogue of Theorems~\ref{thm:sheaves-simple-wall}
and \ref{thm:sheaves-wall}, but without requiring the strong
assumption~\ref{it:smooth-moduli-stack} on the smoothness of the
moduli spaces.

\subsection{}

\begin{remark} \label{rem:DT-wall-y}
  Suppose that the numerical conditions in Theorem~\ref{thm:DT-wall}
  hold for all $N > 0$, i.e. they are equalities instead of
  congruences mod $N$. Then one obtains equalities
  \[ \DT_{-y}^{\Ell/2}(\alpha; +) = \DT_{-y}^{\Ell/2}(\alpha; -) \]
  under only the weaker assumption~\ref{it:DT-moduli-stack-T-proper}
  which ensures both sides are well-defined. This is because both
  sides only have finitely many poles in $\{|y| = 1\}$, so their $y =
  \zeta_N$ specializations are well-defined for all $N \gg 0$.
  Coefficients of the $q$-series on both sides are rational functions
  of $y$, which are therefore equal if and only if they are equal at
  $y = \zeta_N$ for all $N \gg 0$.
\end{remark}

\subsection{}

\begin{proof}[Proof of Theorem~\ref{thm:DT-wall}.]
  In \cite[\S 10.6]{Joyce2021}, Joyce constructs a master space $M$,
  roughly the moduli space of triples $(E, \vec V, \vec \rho)$ where
  $[E] \in \fN_\alpha^{\sst}(0)$ and $(\vec V, \vec \rho)$ is a
  representation of a certain quiver. There is a forgetful morphism
  \[ \pi\colon M \to \fN_\alpha^\sst(0) \]
  which is smooth as a morphism of algebraic stacks. Then {\it
    symmetrized pullback} \cite[\S 2]{Liu2023} \cite[\S 2]{Kuhn2023}
  along $\pi$ of the equivariantly-symmetric obstruction theory on
  $\fN_\alpha^\sst(0)$ (not necessarily perfect!) results in a
  equivariantly-symmetric {\it almost perfect} obstruction theory on
  $M$. By Remark~\ref{rem:APOTs}, this suffices for wall-crossing.

  Although the almost perfect obstruction theory on $M$ can only be
  compared \'etale-locally to the (pullback along $\pi$ of the) original
  obstruction theory on $\fN_\alpha^\sst(0)$, there is a well-defined
  global virtual tangent bundle on $M$, satisfying
  \begin{equation} \label{eq:symmetrized-pullback-Tvir}
    \cT_M^\vir = \pi^*\cT_{\fN_\alpha^\sst(0)}^\vir + (\cT_\pi - y^{-1} \cT_\pi^\vee) \in K_{\tilde\sT}(M)
  \end{equation}
  where $\cT_\pi$ is the relative tangent complex of $\pi$. 

  \subsection{}

  We will apply Theorem~\ref{thm:virtual-invariance}. The analogue of
  \eqref{eq:interaction-locus-normal-bundle},
  \[ \cN_{\iota_0}^\vir\Big|_{[E_1 \oplus E_2]} = -s^{-1} \Ext_Y(E_1, E_2) - s \Ext_Y(E_2, E_1) \] 
  continues to hold. By Serre duality, in the setting of a simple wall,
  \[ \cN_0^{\vir/2}\Big|_{[E_1 \oplus E_2]} = -s^{-1} \Ext_Y(E_1, E_2). \]
  In the setting of a non-simple wall where the wall-crossing problem
  has been lifted to the auxiliary abelian category
  $\tilde{\cat{A}}^{\Fr}$, like in
  \S\ref{sec:mochizuki-auxiliary-category},
  \[ \cN_0^{\vir/2}\Big|_{[E_1 \oplus E_2, V_1^\bullet \oplus V_2^\bullet)]} = -s^{-1}\Ext_Y(E_1, E_2) + \left(s^{-1} \Ext_Q(V_1^\bullet, V_2^\bullet) + s \Ext_Q(V_2^\bullet, V_1^\bullet)\right) \]
  using \eqref{eq:symmetrized-pullback-Tvir}. This explains the
  numerical conditions in Theorem~\ref{thm:DT-wall}.

  \subsection{}
  \label{sec:master-space-Tw-proper}

  Finally, care is required when applying
  Theorem~\ref{thm:virtual-invariance}, because the master space $M$
  is not proper. We will use the argument in
  Remark~\ref{rem:master-space-proper} to work around this issue, by
  checking that all $\tilde\sT_w$-fixed loci of $M$ are proper, for
  maximal tori
  \[ \tilde\sT_w \subset \ker(w) \subset \tilde\sT \times \sS \]
  where $w$ is a $(\tilde\sT \times \sS)$-weight with non-trivial
  $\sS$-component. This is the same argument as in \cite[Lemma
    5.7]{Liu2023}, which we summarize for the sake of completeness.
  Take any $\tilde\sT$-equivariant compactification $\bar Y$ of $Y$.
  Let $\bar\fN$ denote the moduli stack $\fN$ but for objects on $\bar
  Y$, and similarly let $\bar M$ denote the master space for
  $\bar\fN$. It is known that $\bar M$ is proper. Consider the
  inclusions
  \[ M^{\tilde\sT_w} \subset \bar M^{\tilde\sT_w} \subset \bar M. \]
  The second inclusion is clearly closed. The first inclusion is also
  closed: on triples $(E, \vec V, \vec\rho)$ parameterized by $\bar
  M$, only $\tilde\sT \subset \tilde\sT_w$ acts on $E$, and so
  \[ M^{\tilde\sT_w} = \{(E, \vec V, \vec\rho) : \supp E \subset Y^{\tilde\sT} \subset \bar Y^{\tilde\sT}\} \]
  where $\supp$ means set-theoretic support. In other words, the
  $\tilde\sT_w$-fixed locus in $M$ is a collection of certain
  $\tilde\sT_w$-fixed components of $\bar M$. Closed subsets of proper
  spaces are proper.

  Applying Theorem~\ref{thm:virtual-invariance} concludes the proof.
\end{proof}

\subsection{}
\label{sec:vafa-witten-setup}

We give one explicit situation, {\it Vafa--Witten (VW) theory}
\cite{Tanaka2020}, in which the divisibility of $\Ext_Y(E_1, E_2)$,
required by Theorem~\ref{thm:DT-wall}, can be controlled. VW theory is
a form of DT theory when $Y = \tot(\cK_S)$ is an {\it equivariant
  local surface}, meaning:
\begin{itemize}
\item $S$ is a smooth projective surface acted on by a torus $\sT$;
\item $\tilde \sT \coloneqq \sT \times \bC^\times$ where $\bC^\times$
  acts by scaling the fibers of $\pi\colon Y \to S$ with weight
  $y^{-1}$.
\end{itemize}
Let $\fM$ be the moduli stack of {\it compactly-supported} coherent
sheaves on $Y$. By the spectral construction \cite[\S 2]{Tanaka2020},
a point $[\cE] \in \fM$ is equivalent to a pair $(\bar \cE, \phi)$
where
\[ \bar \cE = \pi_*\cE \in \cat{Coh}(S), \quad \phi \in \Hom_S(\bar \cE, \bar \cE \otimes \cK_S). \]

\begin{lemma}[{\cite[Proposition 2.14]{Tanaka2020}}]\label{lem:VW-Ext-parity}
  For $\cE_1, \cE_2 \in \fM$,
  \[ \Ext_Y(\cE_1, \cE_2) = \Ext_S(\bar \cE_1, \bar \cE_2) - y^{-1} \Ext_S(\bar \cE_2, \bar \cE_1)^\vee. \]
\end{lemma}

\subsection{}

The VW moduli stack comes in two flavors. Fix a class $\alpha = (r,
c_1, c_2) \in \bZ_{>0} \oplus H^2(S) \oplus H^4(S)$.
\begin{itemize}
\item ($\mathrm{U}$) If $H^1(\cO_S) = H^2(\cO_S) = 0$, then define
  the moduli substack
  \[ \fN_\alpha \coloneqq \{\ch(\bar \cE) = \alpha\} \subset \fM. \]
\item ($\SU$) Otherwise, pick $\cL \in \Pic(S)$ with $c_1(\cL) = c_1$
  and define the moduli substack
  \[ \fN_\alpha \coloneqq \{\det \bar \cE = \cL, \; \tr \phi = 0, \, \ch(\bar\cE) = \alpha\} \subset \fM. \]
\end{itemize}
Then $\fN_\alpha$ has equivariantly-symmetric perfect obstruction
theory given by $\Ext_Y(\cE, \cE)$ or some reduction of it
\cite[Corollary 2.26, Theorem 5.46]{Tanaka2020}. Assuming
\ref{it:DT-moduli-stack-T-proper},
Definition~\ref{def:elliptic-DT-invariant} produces {\it elliptic VW
  invariants}
\[ \VW^{\Ell/2}_{-y}(\alpha; \sigma) \in K_{\tilde\sT}(\pt)_{\loc}\pseries*{q}. \]
Typically, $\sigma$ is Gieseker stability with respect to a choice of
ample line bundle, but more general $\sigma$ are permitted as long as
the properness assumption~\ref{it:DT-moduli-stack-Ty-proper} holds.

\subsection{}

\begin{corollary}\label{cor:VW-wall-spin}
  If $S$ is a smooth projective surface whose canonical bundle admits
  a square root, then the hypothesis of Theorem~\ref{thm:DT-wall:it:2}
  is satisfied for $N = 2$.
\end{corollary}

This is the direct analogue of Corollary~\ref{cor:sheaves-wall-spin}.

\begin{proof}
  By Lemma~\ref{lem:VW-Ext-parity} and
  Lemma~\ref{lem:surface-Ext-parity:spin}, $\dim \Ext_Y(\cE_1, \cE_2)
  \equiv 0 \bmod{2}$. By the same argument as in the proof of
  Corollary~\ref{cor:sheaves-wall-spin}, the framing functor may be
  doubled so that $\dim \Ext_Q(V_1^\bullet, V_2^\bullet) - \dim
  \Ext_Q(V_2^\bullet, V_1^\bullet) \equiv 0 \bmod{2}$ as well.
\end{proof}

\subsection{}

Finally, in the VW setting, we can simplify the properness
assumption~\ref{it:DT-moduli-stack-Ty-proper}.

\begin{lemma}\label{lem:VW-poles}
  Suppose that $\cK_S\big|_{S^\sT}$ has non-trivial $\sT$-weight on
  each component. Then assumption~\ref{it:DT-moduli-stack-Ty-proper}
  is satisfied.
\end{lemma}

\begin{proof}
  Since $\sT_y = \sT$, by hypothesis the $\sT_y$-fixed locus of $Y$
  lies within $S$ and is therefore proper. Then the $\sT_y$-fixed
  locus of $\fN_\alpha^{\sst}(\sigma)$ is also proper by the same
  argument as in \S\ref{sec:master-space-Tw-proper}: it is a closed
  subspace in the analogous moduli space for any choice of
  compactification $\bar Y$, which is proper for standard reasons
  \cite{Huybrechts2010}.
\end{proof}

\phantomsection
\addcontentsline{toc}{section}{References}

\begin{small}
\bibliographystyle{alpha}
\bibliography{EllRes}
\end{small}

\end{document}